\documentclass[12 pt]{article}

\usepackage{graphicx}

\usepackage[numbers]{natbib}
\usepackage{amssymb}
\usepackage{latexsym}
\usepackage{eqparbox}
\usepackage{amsthm}
\usepackage{amsmath}
\usepackage{thmtools}
\usepackage{lineno}
\usepackage{color}
\usepackage{tikz}
\usepackage{float}
\usepackage{caption}
\usepackage{rotating}
\usepackage{subcaption}
\usepackage{lscape}
\usepackage{multirow}
\usepackage{tikz-3dplot}
\usepackage{threeparttable}
\usepackage{textcomp}
\tdplotsetmaincoords{120}{120}
\usepackage{geometry}
\newtheorem{theorem}{Theorem}[section]
\newtheorem{example}{Example}[section]
\newtheorem{remark}{Remark}[section]
\newtheorem{lemma}{Lemma}[section]
\newtheorem{proposition}{Proposition}[section]
\newtheorem{corollary}[theorem]{Corollary}

\usepackage[T1]{fontenc}
\usepackage[utf8]{inputenc}
\usepackage{authblk}

\title{On Tail Dependence Matrices \\ \Large{The Realization Problem for Parametric Families}}
\author{Nariankadu D. Shyamalkumar\thanks{Corresponding author: N. D. Shyamalkumar; Tel.: +1-319-335-1980} }
\author{Siyang Tao}
\affil{\small{Department of Statistics and Actuarial Science\protect\\ The University of Iowa\protect\\ 241 Schaeffer Hall, Iowa City, IA 52242, USA \protect\\ Emails: shyamal-kumar@uiowa.edu (N. D. Shyamalkumar),  siyang-tao@uiowa.edu (S. Tao)}\vspace{-2ex}} 
\date{}

\def\E#1{\mathbb{E}{\left(#1\right)}}

\def\Pr#1{{\rm Pr}\left(#1\right)}

\def\eqd{\buildrel \rm d \over =}

\def\gbequiv{\buildrel {\mathbf{B}} \over \equiv}

\def\calA{{\mathcal{A}}}

\def\calB{{\mathcal{B}}}
\def\calG{{\mathcal{G}}}

\def\calT{{\mathcal{T}}}

\def\x{(x)}
\def\Real{\mathbb{R}}
\def\iff{if and only if }

\def\myord{ \prec_{{\scriptscriptstyle \mathbf{d}}}}

\begin{document}

\maketitle

\begin{abstract}
Among bivariate tail dependence measures, the tail dependence coefficient has emerged as the popular choice. 
Akin to the correlation matrix, a multivariate dependence measure is constructed using these bivariate measures, and this is referred to in the literature as the tail dependence matrix (TDM). While the problem of determining whether a given ${d\times d}$ matrix is a correlation matrix is of the order $O(d^3)$ in complexity, determining if a matrix is a TDM (the realization problem) is believed to be non-polynomial in complexity. Using a linear programming (LP) formulation, we show that the combinatorial structure of the constraints is related to the intractable max-cut problem in a weighted graph. This connection provides an avenue for constructing parametric classes admitting a polynomial in $d$ algorithm for determining membership in its constraint polytope. The complexity of the general realization problem is justifiably of much theoretical interest. Since in practice one resorts to lower dimensional parametrization of the TDMs, we posit that it is rather the complexity of the realization problem restricted to parametric classes of TDMs, and algorithms for it, that are more practically relevant. In this paper, we show how the inherent symmetry and sparsity in a parametrization can be exploited to achieve a significant reduction in the LP formulation, which can lead to polynomial complexity of such realization problems - some parametrizations even resulting in the constraint polytope being independent of $d$. We also explore the use of a probabilistic viewpoint on TDMs to derive the constraint polytopes.      



\end{abstract}

\section{Introduction}
For a $d$-dimensional random vector the complete description of the dependence between its components is widely accepted to be provided by its copula\footnote{We assume that the the marginals are continuous, which implies a unique copula by Sklar's Theorem.}. The accompanying complexity of complete descriptions make them unsuitable for some purposes. An example of one such is the task of comparing the dependence embodied by two distributions in the same parametric family. 
This becomes a lopsided trade-off when $d$ is large. A solution then is to summarize multivariate dependence in terms of pairwise dependence, with pairwise dependence summarized by a numerical measure as well. In the case of dependence, this has led to various definitions of rank correlations like Kendall's tau and Spearman's rho, which unlike the linear correlation coefficient are functions of the copula. A popular measure of bivariate dependence in the tail is the lower\footnote{We note that without loss of generality we can restrict attention to lower tail dependence as for a copula $C$, $C^\ast(u_1,u_2):=C(1-u_1,1-u_2)$ is a copula as well, and lower tail dependence coefficient w.r.t. $C$ equals the upper tail dependence coefficient w.r.t. $C^\ast$ and vice versa.} tail dependence coefficient. For a random vector $(X_1,X_2)$ with $X_1\sim F_1$ and $X_2\sim F_2$, it is defined as  
\begin{eqnarray*}
\lim_{u \downarrow 0 } \frac{\Pr{F_1(X_1) \leq u, F_2(X_2) \leq u}}{u}, 
\end{eqnarray*} 
given the limit exist; we denote the above limit by $\chi(X_1,X_2)$. This limit need not always exist, see \citet{KortAlb}. We refer to \citet{Fiebig2017} for historical references and related measures used in other fields. The function $\chi(\cdot,\cdot)$, which we will simply refer to as the tail dependence coefficient, is clearly symmetric and takes values in $[0,1]$. Moreover, for $c\in[0,1]$, $I\sim Ber(c)$, $U_1, U_2$ iid $U(0,1)$ random variables independent of $I$, $\chi(U_1,IU_1+(1-I)U_{2})=c$; hence the range of $\chi(\cdot,\cdot)$ is $[0,1]$. 

The matrix of tail dependence coefficients $(\chi(X_i,X_j))_{1\leq i,j\leq d}$ corresponding to a $d$ dimensional random vector $\mathbf{X}_{d\times 1}$ is called a  Tail Dependence Matrix (TDM). The set of all such matrices is denoted by $\calT_d$. As observed in \citet{Strokorb2013} and \citet{Embrechts2016}, determining if a given matrix belongs to $\calT_d$, the tail dependence matrix (TDM) realization problem, is deceptively hard. In \citet{Strokorb2013} and \citet{Fiebig2017} it is shown that $\calT_d$ is essentially a polytope in $\Real^{\binom{d}{2}}$ whose number of vertices and facets grow non-polynomially in $d$. Moreover, a complete description of these polytopes is available only for $d\leq 6$; see \citet{Fiebig2017}.  In contrast, to determine if a matrix is a valid correlation matrix requires essentially checking non-negative definiteness which requires $O(d^3)$ number of operations.  

Our interest in the TDM realization problem arose from \citet{Embrechts2016}, where the design of an algorithm for the TDM realization problem is posed as an open problem of significant interest. One of their main results is the characterization of $\calT_d$ in terms of matrices which can be expressed as $\E{\mathbf{X}\mathbf{X}^{\mathsf{T}}}$ where $\mathbf{X}_{d\times 1}$ is a random vector taking values in $\{0,1\}^d$. Such matrices are called Bernoulli Compatible Matrices (BCMs), and we denote the set of all BCMs by $\calB_{d}$. Since $\Pr{\mathbf{X}_{d\times 1}\in\{0,1\}^d}=1$, $\E{\mathbf{X}\mathbf{X}^{\mathsf{T}}}$ is a convex combination of points in $\{\mathbf{x}\mathbf{x}^\mathsf{T}: \mathbf{x}\in \{0,1\}^d\}$, and hence the convex hull of the latter set, which is a closed convex set, equals $\calB_{d}$. A set of matrices related to $\calB_d$, denoted by $\calB^{\mathsf{I}}_d$, is defined as follows:
\[
\calB^{\mathsf{I}}_d=\{\mathbf{B} \vert \mathsf{diag}(\mathbf{B})=\mathbf{1}_{d\times 1} \hbox{ with } p\mathbf{B} \in \calB_d \hbox{ for some } p\in(0,1] \}.
\]
The above can be written as 
\[
\calB^{\mathsf{I}}_d=\{\mathbf{B} \vert \mathsf{diag}(\mathbf{B})=\mathbf{1}_{d\times 1} \hbox{ with } \mathbf{B} \in \calB_d^\ast  \},
\]
where $\calB_d^\ast$ is the convex cone generated by $\calB_d$. 
In Theorem 3.3 of \citet{Embrechts2016}, it is shown that $\calT_d=\calB^{\mathsf{I}}_d$, establishing a connection with testing membership in $\calB_d$. 

\citet{Fiebig2017} describes the polytope $\calT_d$; but the generality of the problem along with the accompanying exponential increase in complexity of the polytope only permits full description up to dimension six. Note that if there was a polynomial in dimension sized facet representation, then determination of membership will be a trivial exercise. On the other hand, \citet{Krause2018} presents an algorithm for determining the membership of an arbitrary matrix in $\calB_d$, which even though having exponential in dimension worst case performance as any other such algorithm, is able to demonstrate practically acceptable performance for dimensions more than $40$ on their well designed test battery. 
On the other hand, in large dimensional modeling problems, to escape the curse of dimensionality and for parsimony, one resorts to a lower dimensional dependence structure. This in turn results in a subclass of TDMs that has a lower dimensional parametric structure. We show in this paper that for many parametric classes of practical interest the constraint polytope can be derived in polynomial time. Developing techniques that will help in deriving complete descriptions of such constraint polytopes we posit is of much practical import, and this is our main contribution.  

In section 2, we emphasize the probabilistic viewpoint as a technique to derive necessary and/or sufficient conditions for sub-classes of TDMs. In particular, we find necessary and sufficient conditions for a more general version of the motivating example of \citet{Embrechts2016}, and derive a set of sufficient conditions for the sub-class of TDMs resulting from $d$ consecutive observations of a stationary process. Interestingly, this latter set of sufficient conditions is satisfied by three classes of TDMs considered in \citet{Embrechts2016} and another considered in \citet{Falk2005}. In section 3, we point out that while the membership problem in $\calT_d$ is very likely to be NP-hard, as it is for $\calB_d$, it is yet to be settled. While the duals of LP formulations of \citet{Lee1993} and \citet{Krause2018} of the TDM realization problem have exponential in $d$ number of constraints, it is known that in such cases the combinatorial structure of the constraints becomes the determinant of the complexity of the problem. Using the ellipsoid method of \citet{Khachiyan} (see also Chapter 3 of \citet{Grotschel1993}) we show that this combinatorial structure is related to the intractable max-cut problem on an undirected graph with arbitrary weights. This not only suggests the lack of polynomial time algorithms for the TDM realization problem but also suggests that a way to identify sub-classes of TDMs for which the realization problem is polynomial time solvable is to find those that correspond to problems where the resulting separation problem is  polynomial time solvable. In section 4, we show that the size of the LP formulation can be dramatically reduced in some parametric classes of TDMs by exploiting the inherent symmetry and the presence of zero elements (pairwise tail independence) in the parametrization. In Example \ref{ex:equi-correlation}, we present a parametric class in the related BCM realization problem which has two parameters, the facet representation growing linearly with $d$ and the realization problem solvable in constant time - its analysis uses the presence of significant symmetry in the parametrization. This represents a super class of those considered in \citet{Embrechts2016} and \citet{Krause2018}. In Proposition \ref{MA2}, we present a two-parameter subclass of TDMs where the constraint polytope remains the same for $d\geq 6$, and hence the realization problem is solvable in constant time - its analysis use the fact that there is only a linear number of non-zero elements. Finally, we present a three parameter class where the symmetry is used to get a polynomial time LP formulation for the realization problem. In section 5, we present a practically significant enhancement of the algorithm of \citet{Krause2018} by replacing the solution of a NP-hard problem by a continuous relaxation which is in P. Also, we present numerical results confirming the advantage of using LP formulations that employ a non-trivial objective function as that in \citet{Krause2018}. Finally, we make some comments, supported by computational results, contrasting the general approach to membership testing as represented by the algorithms of \citet{Lee1993} and \citet{Krause2018} with a customized approach of this paper. 


{\bf Notation:}  All vectors and matrices are boldfaced, with matrices in uppercase; we may denote their dimensions by using subscripts, especially when it is not clear from the context. We denote by $\mathbf{I}$ the identity matrix. By $\mathbf{1}$ and $\mathbf{J}$ we denote vectors and matrices  with all elements equal to $1$, respectively.  Also, by $\mathbf{0}$ we denote vectors and matrices with all of the elements equal to $0$. Our canonical probability space, denoted by  $([0,1],\mathcal{L},\lambda)$, is one with $\mathcal{L}$ being the set of Lebesgue measurable subsets of $[0,1]$ and $\lambda$ representing the Lebesgue measure on $[0,1]$. We denote the floor function by $\lfloor \cdot \rfloor$, the ceiling function by $\lceil \cdot \rceil$, and the indicator function of a set $A$ by $I_A$. 

\section{An Alternate Viewpoint of BCMs}
In this section our goal is to show that a more probabilistic viewpoint of $\calB_d$ can be quite fruitful in deriving both necessary and sufficient conditions on the parameters for instances in a parametric family to belong to $\calB_d$ or $\calT_d$. We begin with the following result which equates the problem of determining whether a given matrix $\mathbf{T}$ is in $\calT_d$ to that of verifying the membership of $(1/d)\mathbf{T}$ in $\calB_d$. This is essentially a refinement of the result that $\calT_{d}=\calB^{\mathsf{I}}_d$ established in \citet{Embrechts2016}. We state it without proof since it is observed in Remark 17 of \citet{Fiebig2017}, and is Theorem 1 of \citet{Krause2018}. 

\begin{theorem} \label{BCTD}
Let $\mathbf{T}$ be any $d \times d$ matrix with unit diagonal elements. Then $\mathbf{T}\in\calT_d$ if and only if $(1/d)\mathbf{T}\in\calB_d$.
\end{theorem}

Interestingly, some parametrizations of matrices are much easier to analyze for the TDM realization problem in contrast to the BCM realization problem. Below is one such parametric class of matrices, suggested by the equi-correlation matrices.

\begin{example}\label{ex:equi-correlation}
Consider a matrix $\mathbf{B}_{d \times d}$ of the form
	\begin{eqnarray}\label{equi-corr-mat}
		\begin{bmatrix}
			\alpha &\beta & \cdots & \beta \\
			\beta &\alpha & \cdots & \beta \\
			\vdots & \vdots &\ddots &\vdots \\
			\beta &\beta & \cdots & \alpha
		\end{bmatrix} =(\alpha-\beta) \mathbf{I}_{d}+ \beta \mathbf{J}_{d},
	\end{eqnarray}
 where $ 0 \leq  \beta\leq \alpha\leq 1 $. It is of interest to derive necessary and sufficient conditions on $\alpha$ and $\beta$ for the matrix to be a BCM. Note that since the matrix $\alpha \mathbf{J}_{d\times d}$ is a BCM (equals $\E{\mathbf{X}\mathbf{X}^\mathsf{T}}$ with $X_1=X_2=\cdots=X_d$ and $\E{X_1}=\alpha$), and the class of all BCMs is convex, the necessary and sufficient conditions reduce to finding for each $\alpha\in[0,1]$, a lower bound for such values of $\beta$, which we denote by $\beta_l(\alpha)$.     
 
 We note that for TDMs of this parametric form, in view of Theorem \ref{BCTD}, we need necessary and sufficient conditions only for the sub-class of matrices where $\alpha=1/d$, or in view of the above, only $\beta_l(1/d)$. We note that $\beta_l(1/d)=0$ by taking $(X_1,\ldots,X_d)$ to be a $d$-dimensional multinomial random vector with $n=1$ and equal probabilities of $1/d$ for the $d$ outcomes. This, in other words, says that all matrices with unit diagonal and with constant off-diagonal elements are TDMs \iff the off-diagonal elements are in $[0,1]$. An alternate argument is presented in Proposition 4.1 (i) of \citet{Embrechts2016}. 
 
In \citet{Chaganty2006} they study Bernoulli vectors with constant correlation matrices, and in this connection they establish bounds for this constant correlation coefficient. Their parametrization coincides with the above only in the exchangeable case, but their generality restricts them to dimensions up to three. In \citet{Krause2018}, their Problem Classes 1 and 2 are alternately parametrized sub-classes of the above class. For their Problem Class 1, $\beta=\kappa\alpha^2$, and hence their parametrization is in terms of $(\alpha,\kappa)$ in $[0,1]^2$; they computationally graph in their Figure 5 the region of their parameter space for dimensions in $\{3,6,9,12\}$. 
Later, in section $4$, we derive the following expression for $\beta_l(\cdot)$:
\begin{eqnarray*}
		 \beta_l(\alpha)=\frac{ \left(2\alpha d- \lfloor  \alpha d \rfloor-1\right) \lfloor  \alpha d \rfloor}{d(d-1)}. 
	\end{eqnarray*}
\hfill $\qed$
\end{example}

Now we present a rather simple but an alternate viewpoint of matrices in $\calB_d$. Let $\mathbf{B}\in\calB_d$, and $\mathbf{X}=(X_1,...,X_d)$ be a random vector taking values in $\{0,1\}^d$ such that $\mathbf{B}=\E{\mathbf{X}\mathbf{X}^\mathsf{T}}$. Defining the events $A_i$ by $A_i:=\{X_i=1\}$, for $i=1,\ldots d$, we see that 
\begin{eqnarray}\label{BCprob}
	\mathbf{B}=\begin{bmatrix}
			\Pr{A_1}   & \Pr{A_1\cap A_2} & \cdots & \Pr{A_1\cap A_d}  \\
			\Pr{A_1\cap A_2} & \Pr{A_2}   & \cdots & \Pr{A_2\cap A_d}  \\
			\vdots      & \vdots     & \ddots &  \vdots \\
			\Pr{A_1\cap A_d} & \Pr{A_2\cap A_d} & \cdots & \Pr{A_d}  \\
		\end{bmatrix}.
\end{eqnarray}
This immediately leads one to think of the problem of verifying membership in $\calB_d$ in terms of the existence of events $A_i$, $i=1,\ldots, d$, on a probability space with themselves and their binary intersections having the specified probabilities. Also, importantly it makes clear that the structure of the matrix is quite intricate and central to the problem of verifying membership in $\calB_d$. We encapsulate this idea below, as a corollary of Theorem \ref{BCTD}, which is also established in Lemma 16 of \citet{Fiebig2017}. 

\begin{corollary}\label{BCprobcor}
Let $\mathbf{T}$ denote a $d \times d$ matrix with unit diagonal elements and $\mathbf{B}$ a $d \times d$ matrix. Then we have the following:
\begin{enumerate}
\item $\mathbf{T}$ is a TDM \iff on a probability space there exists events $A_i$, $i=1,\ldots, d$, such that the matrix given in \eqref{BCprob} equals $(1/d)\mathbf{T}$. 
\item $\mathbf{B}$ is a BCM \iff on a probability space there exists events $A_i$, $i=1,\ldots, d$, such that the matrix given in \eqref{BCprob} equals $\mathbf{B}$. 
\end{enumerate}
\end{corollary}

The following two propositions show that the above viewpoint can lead to simpler arguments in problems of verifying membership of matrices in $\calB_d$. Proposition \ref{ex:Embrechts1} below, generalizes Proposition 4.7 of \citet{Embrechts2016} which deals with their motivating example. 
In Proposition 4.1 of \citet{Embrechts2016}, they show that TDMs with certain parameterizations borrowed from that for correlation matrices generate valid TDMs under certain restrictions on the parameters. In Proposition \ref{Toep}, we show that a certain subset of the class of Toeplitz matrices is in $\calT_d$; this subset includes, as special cases, the parameterizations considered in Proposition 4.1 of \citet{Embrechts2016}.

\begin{proposition}\label{ex:Embrechts1}
A matrix $\mathbf{B}_{d\times d}$ of the form 
\begin{eqnarray*}
\begin{bmatrix}
	\beta_1	&	0	&	\cdots	&	0	&	\alpha_1 \\
	0	&	\beta_2	&	\cdots	&	0	&	\alpha_2 \\
	\vdots	&	\vdots	&	\ddots	&	\vdots & \vdots \\
	0	&	0	&	\cdots	&	\beta_{d-1}	&	\alpha_{d-1} \\
	\alpha_1	&	\alpha_2	&	\cdots	&	\alpha_{d-1}	&	\beta_d\\ 
\end{bmatrix}
\end{eqnarray*}
is in $\calB_d$ if and only if $0 \leq \alpha_i\leq \beta_i$, for $i=1,\ldots, d-1$,  $\sum_{i=1}^{d-1} \alpha_i\leq \beta_d$, and $$\sum_{i=1}^{d} \beta_i-\sum_{i=1}^{d-1} \alpha_i\leq 1.$$
\end{proposition}

\begin{proof}
 First, we show the only if part. Using the notation presented in (\ref{BCprob}), we see that the sets $A_i$, for $i=1,\ldots,d$, corresponding to $\mathbf{B}$ satisfy 
\begin{alignat*}{2}
&\Pr{A_i}=\beta_i,&	&\quad \hbox{for }	i=1,\ldots, d; \\
&\Pr{A_j \cap A_d}=\alpha_j,&	&\quad\hbox{for }	j=1, \ldots, d-1;  \\
&\Pr{A_i \cap A_j}=0,&	&\quad\hbox{for }	1 \leq i<j \leq d-1.
\end{alignat*}
By using the above relations, we have from axioms of probability that if $\mathbf{B}\in\calB_d$ then 
$0 \leq \alpha_i\leq \beta_i$, for $i=1,\ldots,d-1$, 
\[
\sum_{i=1}^{d-1} \alpha_i= \sum_{i=1}^{d-1} \Pr{A_i \cap A_d} \leq \Pr{A_d}= \beta_d
\]
and 
\[
\sum_{i=1}^{d} \beta_i-\sum_{i=1}^{d-1} \alpha_i= \sum_{i=1}^{d} \Pr{A_i}  -\sum_{i=1}^{d-1} \Pr{A_i \cap A_d}= \sum_{i=1}^{d-1} \Pr{A_i \cap A_d^\mathsf{c}} + \Pr{A_d}\leq 1.
\]

For the if part, let $B_i$, for $i=1,\ldots,d-1$, be disjoint sub-intervals of $(1-\beta_d,1)$ with lengths $\alpha_i$ on the probability space $([0,1],\mathcal{L},\lambda)$. This can be done as $\alpha_i\geq 0$ for ${i=1,\ldots,d-1}$, $\sum_{i=1}^{d-1} \alpha_i\leq \beta_d$ and $\sum_{i=1}^{d} \beta_i-\sum_{i=1}^{d-1} \alpha_i\leq 1$. Also, let $C_i$ be the intervals $$\left(\sum_{j=1}^{i-1}\beta_j-\sum_{j=1}^{i-1}\alpha_j, \sum_{j=1}^{i}\beta_j-\sum_{j=1}^{i}\alpha_j \right), \quad \hbox{for } i=1,\ldots, d-1;$$ this is possible because of $0 \leq \alpha_i\leq \beta_i$, for $i=1,\ldots, d-1$. Now by defining $$ A_i:=B_i\cup C_i, \hbox{for } {i=1,\ldots, d-1},$$ and ${A_d:=(1-\beta_d,1)}$, we immediately see that the matrix given in \eqref{BCprob} equals $\mathbf{B}$; whence $\mathbf{B}\in\calB_d$. 
\end{proof}

\begin{remark}
If we instead consider TDMs of the above parametric form, {\it i.e.} a matrix $\mathbf{T}_{d\times d}$ of the form 
\begin{eqnarray*}
\begin{bmatrix}
	1	&	0	&	\cdots	&	0	&	\alpha_1 \\
	0	&	1	&	\cdots	&	0	&	\alpha_2 \\
	\vdots	&	\vdots	&	\ddots	&	\vdots & \vdots \\
	0	&	0	&	\cdots	&	1	&	\alpha_{d-1} \\
	\alpha_1	&	\alpha_2	&	\cdots	&	\alpha_{d-1}	&	1\\ 
\end{bmatrix},
\end{eqnarray*}
then the above proposition along with Theorem \ref{BCTD} implies that $\mathbf{T}\in\calT_d$ if and only if $\alpha_i\geq 0$, for $i=1,\ldots, d-1$, and $\sum_{i=1}^{d-1} \alpha_i\leq 1$.
\end{remark}

The next class of TDM matrices is motivated by looking at the $d$-dimensional marginal of a stationary process. Since our interest is only on the TDM, we require only the invariance of the $d$-dimensional TDM with respect to shifts. The {$d$-dimensional} TDMs for such processes is a Toeplitz matrix, and in the following we will denote a general such $d$-dimensional Toeplitz matrix by
\begin{equation}
\begin{bmatrix}
1 & \alpha_1 & \alpha_2 & \cdots & \alpha_{d-2} & \alpha_{d-1} \\
\alpha_1 & 1 & \alpha_1 & \ddots &  & \alpha_{d-2} \\
\alpha_2 & \alpha_1 & \ddots & \ddots & \ddots& \vdots \\
\vdots & \ddots &\ddots & \ddots & \alpha_1 &\alpha_2 \\
\alpha_{d-2} &  &\ddots & \alpha_1 & 1 & \alpha_1\\
\alpha_{d-1} & \alpha_{d-2} & \cdots &\alpha_2 & \alpha_1 &1
\end{bmatrix}. \label{Toep_matrix}
\end{equation}
It is natural then to seek conditions on $\alpha_i$, $i=1,\ldots,d-1$, under which a given matrix of the form \eqref{Toep_matrix} is a TDM; Proposition \ref{Toep} below provides a set of sufficient conditions. It is noteworthy that not all positive definite Toeplitz matrices are TDMs. Consider the positive definite matrix given by  
\[
\frac{1}{3} \begin{bmatrix}
3 & 2 & 0\\
2 & 3 & 2 \\
0 & 2 & 3 
\end{bmatrix}.
\]
Assuming that the above matrix is a TDM, Corollary \ref{BCprobcor} implies that we have three events $A_1, A_2$ and $A_3$ with $\Pr{A_1}=\Pr{A_2}=\Pr{A_3}=1/3$ and satisfying 
\[
\Pr{A_1 \cap A_2}=\Pr{A_2 \cap A_3}=2/9, \hbox{ and } \Pr{A_1 \cap A_3}=0.
\]
However, this implies that 
\[
\Pr{A_2} \geq \Pr{A_1 \cap A_2}+\Pr{A_2 \cap A_3} - \Pr{A_1\cap A_2 \cap A_3}=4/9,
\]
a contradiction. Hence the above matrix is not a TDM.

\begin{proposition} \label{Toep}
A Toeplitz matrix parametrized as given in \eqref{Toep_matrix} is a $d$-dimensional TDM if 
\begin{eqnarray}\label{Toepcond}
\begin{cases}
\alpha_{d-1} \geq 0; \\
1-\alpha_1 \geq \alpha_1-\alpha_2 \geq \cdots \geq \alpha_{d-2}-\alpha_{d-1}\geq 0.
\end{cases}
\end{eqnarray}
\end{proposition}
\begin{proof}
By Theorem \ref{BCTD}, we need to prove that $(1/d)\mathbf{T} \in \mathcal{B}_d$. For convenience, we define $\alpha_0=1$. Let the probability space be $([0,1],\mathcal{L},\lambda)$. Let $B_{i,1}$ be the intervals $[0,\alpha_{i-1}/d)$, $i=1,...,d$. Also, for $i=2,...d$, let $B_{i,j}$ be the interval $$[(j-1)/d,(j-1+\alpha_{i-j}-\alpha_{i-j+1})/d],\quad j=2,\ldots, i.$$ Now defining $A_i=\cup_{j=1}^i B_{i,j}$, for $i=1,\ldots, d$, we see that the matrix given in \eqref{BCprob} equals $(1/d)\mathbf{T}$; whence $(1/d)\mathbf{T}\in\calB_d$. 
\end{proof}


In \citet{Embrechts2016}, they consider three parametric classes suggested by the equi-corrleation matrices, stationary AR(1) autocorrelation matrices, and stationary $1$-dependent autocorrelation matrices. The uni-parameter class suggested by equi-correlation matrices consists of matrices of the form given in \eqref{equi-corr-mat} with $\alpha=1$ and $\beta\geq 0$. This class clearly is a subset of Toeplitz matrices which satisfy \eqref{Toepcond} if $\beta\in[0,1]$. A stationary AR(1) process with parameter $\phi$ has a geometrically decaying autocorrelation, which suggests the Toeplitz parametric class of the form \eqref{Toep_matrix} with 
\[
\alpha_k=\phi^{k}, \quad k=1,\ldots,d-1.
\]
It is easily checked that \eqref{Toepcond} is satisfied for such matrices with $\phi\in[0,1]$, and hence they are TDMs as well. Finally, in the case of a stationary $1$-dependent process, it is clear that the autocorrelation matrix of $d$ successive observations is a matrix of the form \eqref{Toep_matrix} with 
\[
\alpha_1=\alpha;\; \alpha_k=0, \quad k=2,\ldots,d-1.
\]
Such matrices trivially satisfy \eqref{Toepcond} \iff $0\leq \alpha\leq 1/2$. 

In \citet{Falk2005} (see Theorem 1.3 therein), it is proved that any $d\times d$ matrix $\mathbf{T}=(t_{ij})$ satisfying
\[
\text{diag}(\mathbf{T})=\mathbf{1} \; \text{ and } \; t_{kl}=2-(1+|a_l-a_k|)^{1/\theta},\quad  k \neq l \in \{1,\ldots, d \},
\]
for $\theta \geq 1$ and $0 \leq a_1 \leq \cdots \leq a_d \leq 1$ is a TDM of a random vector with margins arbitrarily chosen from among the reverse Weibull, Gumbell and Fr\'echet univariate distributions. Note that any such TDM is in the Toeplitz form \iff $a_i$, $i=1,\ldots,d$, form an arithmetic progression. If we let $a_i=a+b(i-1)$, then the conditions above require $a,b\geq 0$ and  $a+b(d-1)\leq 1$. The matrix $\mathbf{T}$ equals that in \eqref{Toep_matrix} with the off-diagonal elements
\[
\alpha_i=2-(1+bi)^{1/\theta}, \qquad i=1,\ldots, d-1.
\]
Since $0\leq b(d-1)\leq 1$, $\alpha_i$ is decreasing in $i$, $i=1,\ldots,d-1$, and all of these values are contained in $[0,1]$. By defining $\alpha_0=1$, we have $$\alpha_i - \alpha_{i+1}=(1+b(i+1))^{1/\theta} - (1+bi)^{1/\theta},\quad i=0,\ldots,d-2,$$ and $\alpha_i - \alpha_{i+1}$ is decreasing in $i$ as $(1+bx)^{1/\theta}$ is concave in $x$ for $\theta\geq 1$. Hence, our sufficient conditions in \eqref{Toepcond} are satisfied. The above examples suggest that the sufficient conditions in Proposition \ref{Toep} are broad enough to be quite useful in practice.

We note that a $d$-dimensional TDM of the form \eqref{Toep_matrix} cannot necessarily be extended beyond $d$ dimensions while maintaining the $d$-dependent structure. In other words, while for some particular values of $\alpha_1,\ldots,\alpha_{d-1}$  \eqref{Toep_matrix} could be a valid TDM, it does not guarantee that so will the following matrix for some $\alpha_d\in[0,1]$:
\[
\begin{bmatrix}
1 & \alpha_1 & \alpha_2 & \cdots & \alpha_{d-2} & \alpha_{d-1} & \alpha_d\\
\alpha_1 & 1 & \alpha_1 & \ddots &  & \alpha_{d-2} & \alpha_{d-1}\\
\alpha_2 & \alpha_1 & \ddots & \ddots & \ddots& \vdots & \vdots \\
\vdots & \ddots &\ddots & \ddots & \alpha_1 &\alpha_2 & \alpha_3 \\
\alpha_{d-2} &  &\ddots & \alpha_1 & 1 & \alpha_1  & \alpha_2\\
\alpha_{d-1} & \alpha_{d-2} & \cdots &\alpha_2 & \alpha_1 &1 & \alpha_1 \\
\alpha_d& \alpha_{d-1} & \alpha_{d-2} & \cdots &\alpha_2 & \alpha_1 &1  
\end{bmatrix}.
\]
This is observed in the two-dependence settings, as mentioned in Remark \ref{ma2-d<6} below. But it is worth mention that in the case that $\alpha_{d-1}\leq \alpha_{d-2}/2$, Proposition \ref{Toep} guarantees extension of the $d$-dependence structure to any dimension beyond $d$. 
\begin{remark}
We note that the conditions on the parameters given for the three parametric classes from \citet{Embrechts2016} considered above were independent of $d$. In the case of the parametric classes corresponding to the equi-correlation and AR(1) autocorrelation matrices, the stated conditions are clearly necessary as well. In contrast, as shown later in Proposition \ref{MA2} and Remark \ref{ma2-d<6}, in the 1-dependent case while $\alpha\in[0,1/2]$ is necessary and sufficient for $d\geq 3$, it is only sufficient when $d=2$.
\end{remark}

\begin{remark}
We note that akin to an AR(1) process, one could define a process $\{X_k\}_{k\in\mathbb{Z}}$ for $\phi\in[0,1)$ as specified by 
\[
X_k = \phi X_{k-1} \vee Z_k, \quad k\in \mathbb{Z},
\]
where $Z_k$, for $k\in\mathbb{Z}$, are i.i.d. unit Fr\'{e}chet random variables, with cumulative distribution function given by $\exp(-1/x)$, for $x\geq 0$. It is easily verified that the stationary marginal distribution is also unit Fr\'{e}chet with scale parameter $1/(1-\phi)$, and the tail dependence matrix coincides with the autocorrelation matrix of an AR(1) process. Since the maximum of $n$ copies of the process $X$ has the same finite dimensional distribution as $nX$, this process is a simple max-stable process, see \citet{Strokorb2013}. Finally, note that the existence of such a process is an alternate proof of the sufficiency of $\phi\in[0,1)$ for autocorrelation  matrices of AR(1) processes to be TDMs. 
\end{remark}

We note that the second condition in \eqref{Toepcond} is not necessary even in the presence of the monotonicity constraint $1\geq \alpha_1\geq \cdots \geq \alpha_{d-1}\geq 0$. This is seen by considering the matrix 
\[
\begingroup
\renewcommand*{\arraystretch}{1.2}
\begin{bmatrix}
1 & \alpha & \beta & 0\\
\alpha & 1 & \alpha & \beta \\
\beta & \alpha & 1& \alpha\\
0 & \beta & \alpha & 1
\end{bmatrix},
\endgroup
\]
which is shown later in Remark \ref{ma2-d<6} to be a TDM \iff $\alpha$ and $\beta$ satisfy 
\[
\alpha \geq 0,\; \beta \geq 0,\; \alpha+\beta \leq 1, \hbox{ and } 2 \alpha-\beta \leq 1. 
\]
In contrast, \eqref{Toepcond} requires $\alpha \geq 2 \beta$, which is not a necessary condition for it to be a TDM. That the conditions in \eqref{Toepcond} are not necessary is not surprising since the results in Corollary 3.5.8. of \citet{Strokorb2013} and \citet{Fiebig2017} suggest that the polytope of vectors $(\alpha_1,\ldots,\alpha_{d-1})$ which generate a TDM of the form \eqref{Toep_matrix} will have far more than linear in $d$ number of facets.

\section{A Related Combinatorial Optimization Problem}

We have seen in Theorem \ref{BCTD} that the realization problem of $\calT_d$ can be solved using any algorithm for the realization problem of $\calB_d$. Moreover, \citet{Pitowsky1991} established that the decision problem of membership in the correlation polytope is NP-complete; but as observed  in \citet{Fiebig2017} and Corollary \ref{BCprobcor} above, the decision of membership in $\calB_d$ is identical to the latter, and hence is also NP-complete. From this latter fact and Theorem  \ref{BCTD}, one can conclude that the decision problem of membership in $\calT_d$ is in NP. This is so as Carath\'eodory's theorem guarantees a quadratic in $d$ length vertex-representation for any Bernoulli matrix. Hence a  certificate  consists of a quadratic in $d$ number of weights $w_i$ and an equal number of $d$-vectors $\mathbf{v_i}$. This certificate is trivially polynomial-time verifiable as it involves checking membership of $w_i$ and $\mathbf{v_i}$ in $[0,1]$ and $\{0,1\}^d$, respectively, and moreover verifying that the given matrix equals $\sum w_i \mathbf{v_i}\mathbf{v_i}^\mathsf{T}$. But while one would expect because of the similarity of $\calT_d$ with $\calB_d$, and the exponentiality of the facet representation of $\calT_d$ that the decision problem of membership in $\calT_d$ is NP-hard as well, this has not been so far established. The main goal of this section is to make a connection via the ellipsoid algorithm to a combinatorial optimization problem which not only lends heuristic support to the NP-hardness of the general $\calT_d$ realization problem but, importantly, also suggests avenues of identifying subsets of $\calT_d$ such that the restricted membership problem is in P. For developing this connection we begin by defining the notation behind a linear programming formulation of the realization problem motivated by the vertex representation, as done in \citet{Lee1993}, and its dual. 

In the following, we find it convenient to alternately denote a set $A$ by $A^1$ and its complement by $A^0$. Let $\calA:=\{A_1,\ldots,A_d\}$, where $A_i$, for $i=1,\ldots, d$, are $d$ arbitrary events on a probability space $(\Omega, \mathcal{F},P)$ with $\mathbf{X}$ defined as $X_i:=I_{A_i}$, for $i=1,\ldots, d$. Let $q_{i_1,...,i_d}$, for $i_1,...,i_d \in \{0,1\}$, be defined as 
\begin{equation}\label{qdef1}
q_{i_1,...,i_d}:=\Pr{A_1^{i_1}\cap \cdots \cap A_d^{i_d}}=\Pr{\cap_{j=1}^d \{X_j=i_j\}} ,
\end{equation}
and we define the $2^d\times 1$ vector $\mathbf{q}_d$ (or $\mathbf{q}_d(\calA)$ to make the dependence on $\calA$ explicit) by arranging the elements in increasing order of $\sum_{j=1}^d i_j$ and reverse lexicographic order of $(i_1,\ldots,i_d)$. We will denote this latter order on $\{0,1\}^d$ by $\myord$. So for $d=3$, $\mathbf{q}_d$ is defined as 
\begin{equation}\label{qdef2}
(q_{0,0,0},q_{1,0,0},q_{0,1,0},q_{0,0,1},q_{1,1,0},q_{1,0,1},q_{0,1,1},q_{1,1,1})^{\mathsf{T}}.
\end{equation}
Probabilities of the sets in the partition generated by $\calA$ are the coordinates of $\mathbf{q}_d(\calA)$, and hence, in particular, their sum equals $1$. Let the sets in this partition, in the above ordering, be denoted by $A^\ast_i$, for $i=1,\ldots,2^d$; note that we allow for some of these sets to possibly be empty. In view of the representation in (\ref{BCprob}), we define the vector $\mathbf{p}_d$ (or $\mathbf{p}_d(\calA)$) as follows:
\begin{equation}\label{defn:pd}
\mathbf{p}_d:=\left(1,\Pr{A_1},\ldots,\Pr{A_d},\Pr{A_1\cap A_2},\Pr{A_1\cap A_3},\ldots,\Pr{A_{d-1}\cap A_d}\right)^{\mathsf{T}}. 
\end{equation}
For ease in presentation, we have introduced in $\mathbf{p}_d$ the unit first element. Also, Let the sets $A^{\ast\ast}_i$, for $i=1,\ldots,$ $d(d+1)/2+1$, be defined as $A^{\ast\ast}_1=\Omega,A^{\ast\ast}_2=A_1, A^{\ast\ast}_3=A_2,\ldots,A^{\ast\ast}_{d+1}=A_d,A^{\ast\ast}_{d+2}=A_1\cap A_2,\ldots$, so that the i-th element of $\mathbf{p}_d$ equals $\Pr{A^{\ast \ast}_i}$. 
 
The linear map $\mathbf{q}_d(\calA)\mapsto \mathbf{p}_d(\calA)$, a $\left(d(d+1)/2+1 \right)\times 2^d$ matrix denoted by  $\mathbf{C}_d=(c_{ij})$, can be defined independently of the specification of $\Omega,\calA$ - in other words, this linear map is solely a function of $d$.  Such a definition of $\mathbf{C}_d$ is given by  
\begin{eqnarray*}
{c}_{ij}=
\begin{cases}
			1, & A^\ast_{j}\subseteq A^{\ast\ast}_i;\\
			0, & \text{otherwise},
\end{cases}
\end{eqnarray*}	
where $A_i^\ast$ and $A_i^{\ast\ast}$ correspond to the following specification of $\Omega$ and $A_i$'s: $\Omega$ is taken to be $\{0,1\}^d$, and $A_i$ is taken to be $\{0,1\}^{i-1}\times \{1\} \times \{0,1\}^{d-i}$, for $i=1,\ldots,d$. As an example, the coefficient matrix $\mathbf{C}_3$ equals 
\[
\begin{bmatrix}
1 &1 &1&1&1&1&1&1\\
0&1&0&0&1&1&0&1\\
0&0&1&0&1&0&1&1\\
0&0&0&1&0&1&1&1\\
0&0&0&0&1&0&0&1\\
0&0&0&0&0&1&0&1\\
0&0&0&0&0&0&1&1\\
\end{bmatrix}.
\]
It is easy to check that the above defined $\mathbf{C}_d$ satisfies  $\mathbf{C}_d\mathbf{q}_d(\calA)=\mathbf{p}_d(\calA)$ for all specifications of $\calA$. The discussion above leads to the following theorem in which, similar to the definition of $\mathbf{p}_d(\calA)$ above, for any $d\times d$ matrix $\mathbf{B}=(b_{ij})$, we define $\mathbf{p}_d(\mathbf{B})$ as 
\begin{equation}\label{defn:pdB}
\mathbf{p}_d(\mathbf{B}):=\left(1,b_{1,1},\ldots,b_{d,d},b_{1,2},b_{1,3},\ldots,b_{d-1,d}\right)^{\mathsf{T}}. 
\end{equation}

\begin{theorem} \label{FABC}
For a matrix $\mathbf{B}_{d \times d}$ and $\mathbf{p}_d:=\mathbf{p}_d(\mathbf{B})$, the following statements are equivalent:
\begin{enumerate}
\item $\mathbf{B}_{d \times d}$ is in $\calB_d$ 
\item There exists a vector $\mathbf{x} \geq \mathbf{0}$ such that $\mathbf{C}_d\mathbf{x}=\mathbf{p}_d$
\item The optimal value of the following linear programming problem is non-negative:
\begin{equation}\label{LPProblem}
\begin{array}{ll@{}}
\text{minimize}  & \mathbf{p}_d^\mathsf{T}\mathbf{y};\\
\\
\text{subject to}& \mathbf{C}_d^\mathsf{T}\mathbf{y}\geq \mathbf{0}. \\
\end{array}
\end{equation}
\end{enumerate}
\end{theorem}
\begin{proof}
\begin{description}
\item [{\it i.} implies {\it ii.}:] Since $\mathbf{B}$ is in $\calB_d$, there exists a Bernoulli random vector $\mathbf{X}$ such that $\mathbf{B}=\E{\mathbf{X}\mathbf{X}^\mathsf{T}}$. The vector $\mathbf{q}_d$ defined in \eqref{qdef1} that corresponds to $\mathbf{X}$ satisfies 
$\mathbf{q}_d \geq \mathbf{0}$ such that $\mathbf{C}_d\mathbf{q}_d=\mathbf{p}_d$. We note that this {\em vertex representation} based LP formulation was observed in \citet{Lee1993}.
\item [{\it ii.} implies {\it i.}:] Since $\mathbf{x} \geq \mathbf{0}$ and $\mathbf{x}\cdot \mathbf{1}=1$, we partition $[0,1]$ into $2^d$ intervals such that their lengths equal the coordinates of $\mathbf{x}$. Let these intervals be denoted by\\ $(A_{0,...,0},A_{0,...,0,1},...,A_{1,...,1})$. Then the sets $A_i$, for $i=1,\ldots d$, defined by 
\[
A_i=\cup_{\{\mathbf{k} \in\{0,1\}^d | i\hbox{-th coordinate of $\mathbf{k}$ equals } 1\}} A_{\mathbf{k}}
\]
form $d$ events on the probability space $([0,1],\mathcal{L},\lambda)$. These events satisfy \eqref{BCprob}, and hence by Corollary \ref{BCprobcor} we have $\mathbf{B}$ is a BCM.
\item[{\it ii.} is equivalent to {\it iii.}:] Follows directly from Farkas Lemma (for example, see Theorem 0.1.41 of \citet{Grotschel1993}). 
\end{description}
\end{proof}

The above theorem in particular says that the BCM realization problem is equivalent to a LP problem with $d(d+1)/2+1$ variables  and $2^d$ constraints. Theorem \ref{FABC} and Theorem \ref{BCTD} combined lead to a method for verifying the membership of a matrix in $\mathcal{T}_d$, and this is summarized by the following corollary. 

\begin{corollary} \label{Algo1}
Let $\mathbf{T}=(t_{ij})$ be any $d \times d$ matrix with unit diagonal elements. Then $\mathbf{T}\in\calT_d$ if and only if the optimal value of the linear programming problem in \eqref{LPProblem} is non-negative with 
\[
\mathbf{p}_d:=\left(1,1/d,\ldots,1/d,t_{1,2}/d,t_{1,3}/d,\ldots,t_{d-1,d}/d\right)^{\mathsf{T}}.
\]
\end{corollary}

Note that the application of Farkas lemma converts the  problem of verifying the membership of a given $d\times d$ matrix in $\mathcal{B}_d$, equivalently $\mathcal{T}_d$, to an optimization problem in a polyhedra in $\Real^{d(d+1)/2+1}$ as given in \eqref{LPProblem}. Although the polyhedra has an exponential-in-$d$ sized description, this does not preclude the existence of a polynomial time algorithm to solve this optimization problem. This is so as the ellipsoid algorithm (see Chapter 3 in \citet{Grotschel1993})  can yield a polynomial (in $d$) time algorithm if the separation oracle runs in polynomial time. An example of a problem where this happens is the maximum matching problem on a general graph; here the number of constraints is exponential in the number of vertices (see Theorem P in \citet{Edmonds1965}) but nevertheless the separation oracle can be constructed to run in polynomial time using polynomial time algorithms for minimum T-odd cut problem. So it is natural to ask if this happens with the realization problem as well. Towards this end, we first note that in view of Theorem \ref{FABC} {\it iii.}, the realization problem is equivalent to determining the non-emptiness of the following polyhedra:
\begin{equation}\label{ell:eqn1}
\begin{bmatrix}
\mathbf{C}_d^\mathsf{T}\vspace{.1cm}\\
-\mathbf{p}_{d}^\mathsf{T}
\end{bmatrix}\mathbf{y}\geq 
\begin{bmatrix}
\mathbf{0}\vspace{.1cm}\\
1
\end{bmatrix}.
\end{equation}
For ellipsoid algorithm to run in polynomial time, we need, for a given $\mathbf{y}$, to ascertain in polynomial time if \eqref{ell:eqn1} is satisfied, and if not, determine a row $\mathbf{c}^\mathsf{T}$ of   $\mathbf{C}_d^\mathsf{T}$ such that $\mathbf{c}^\mathsf{T}\mathbf{y}<0$. To do this in polynomial time, we need an alternate polynomial in $d$ representation of the constraints, and for this, as is commonly done, we rephrase it as a combinatorial optimization problem. Towards doing this, let $K_{d+1}$ denote the $d+1$ vertex complete undirected weighted graph with vertex set $\{0,1,...,d\}$, and edge weights as described below: the edges $(0,i)$ have weights $y_{i+1}$, for $i=1,\ldots,d$; the edges $(i,j)$, $1\leq i<j \leq d$, in lexicographic ordering have weights $y_k$, $k=d+2,\ldots ,d(d+1)/2+1$, respectively. We  alternately, for convenience, denote the weights on an edge $e$ by $y_e$. It follows with little thought on the structure of $\mathbf{C}_d$ that $\mathbf{C}_d^\mathsf{T}\mathbf{y}\geq \mathbf{0}$ can be represented as 

\begin{eqnarray*}
	\min_{\substack{ S \subseteq \mathbf{S} \\ 0 \in S }} \sum_{e \in E(S) } y_e \geq -y_1,
\end{eqnarray*}
where $y_1 \geq 0$, $\mathbf{S}=\{0,...,d\}$ and $E(S)$, for $S \subseteq \mathbf{S}$, is the set of edges with both endpoints in $S$.
For any undirected weighted graph with a sub-vertex set $S$ and edge weights $y_e$, we have
\begin{eqnarray*}
		\sum_{v \in S} \sum_{e \in \delta(v)} y_e=2\sum_{e \in E(S)} y_e+ \sum_{e \in \delta(S)} y_e,
\end{eqnarray*}
where for a vertex $v$,  $\delta(v)$ represents the set of all edges with $v$ as one of the vertices, and for a sub-vertex set $S$, $\delta(S)$ represents all edges with exactly one vertex in $S$. It follows now that 
\[
\min_{\substack{ S \subseteq \mathbf{S} \\ 0 \in S }} \sum_{e \in E(S) } y_e  = 
\frac{1}{2}\min_{\substack{ S \subseteq \mathbf{S} \\ 0 \in S }} \left(\sum_{v \in S} \sum_{e \in \delta(v)} y_e - \sum_{e \in \delta(S)} y_e\right).
\]
By defining $z_i:= -\sum_{e \in \delta(i)} y_e$, for $i=0,\ldots,d$, the above can be written as 
\[
\min_{\substack{ S \subseteq \mathbf{S} \\ 0 \in S }} \sum_{e \in E(S) } y_e  = 
-\frac{1}{2}\max_{\substack{ S \subseteq \mathbf{S} \\ 0 \in S }} \left(\sum_{v \in S} z_v + \sum_{e \in \delta(S)} y_e\right).
\]
We now define  a new  undirected weighted graph $G$ from $K_{d+1}$ by adding a new vertex $\infty$ to the vertex set, and by adding edges $e=(i,\infty)$ with weights $y_e=z_i$, for $i=0,\ldots,d$ (see Figure \ref{graph:G}). Denoting $\mathbf{S}'= \mathbf{S} \cup \{\infty\}$, the vertex set of $G$, we have that $\mathbf{C}_d^\mathsf{T}\mathbf{y}\geq \mathbf{0}$ is equivalent to the following statement with respect to the graph $G$:
\[
\max_{\substack{ S \subseteq \mathbf{S}' \\ 0 \in S \\ \infty \notin S }} \sum_{e \in \delta(S)} y_e \leq 2y_1.
\]
The above basically says that the maximum weight of a $0$-$\infty$ cut is at most $2y_1$. In general max-cut problems are quite intractable in the sense of being NP-complete (see \citet{Karp1972} and \cite{Garey}). We note that NP-completeness of the latter while suggesting the same for the TDM realization problem, falls short of being a proof. 

\tikzstyle{place}=[circle,draw=black!50,fill=black!20,thick,
                   inner sep=0pt,minimum size=5mm]
    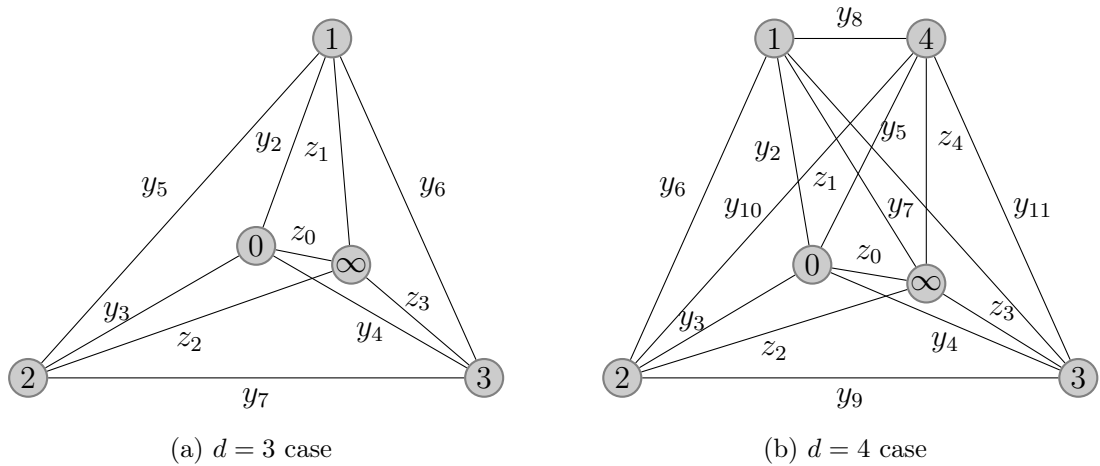
\begin{figure}[H]
 \captionsetup[subfigure]{font=footnotesize}
\centering
\subcaptionbox{$d=3$ case}[.48\textwidth]{    
    \begin{tikzpicture}[scale=0.5]
  \node(1) at ( 11,9) [place] {$1$};
  \node(2) at ( 3,0) [place] {$2$};
  \node(3) at ( 15,0) [place] {$3$};
  \node(0) at ( 9,3.5) [place] {$0$};
  \node(i) at ( 11.5,3) [place] {$\infty$};
  \draw (1) to node[above left]   {$y_5$} (2)
            to  node[below] {$y_7$} (3)
  			to node[above right]{$y_6$}(1);
  	\draw (1) to node[left]   {$y_2$} (0)
  			 to	 node [left]   {$y_3$} (2);
  \draw (1) to node[left]   {$z_1$} (i)
  			 to	 node [below]   {$z_2$} (2);
   	\draw (0) to node[below]   {$y_4$} (3)
  			 to	 node [above]   {$z_3$} (i)
  			 to	 node [above]   {$z_0$} (0);
\end{tikzpicture}
}
\captionsetup[subfigure]{font=footnotesize}
\centering
\subcaptionbox{$d=4$ case}[.48\textwidth]{
\begin{tikzpicture}[scale=0.5]
  \node(1) at ( 7,9) [place] {$1$};
  \node(2) at ( 3,0) [place] {$2$};
  \node(3) at ( 15,0) [place] {$3$};
  \node(4) at ( 11,9) [place] {$4$};
  \node(0) at ( 8,3) [place] {$0$};
  \node(i) at ( 11,2.5) [place] {$\infty$};
  \draw (1) to node[above left]   {$y_6$} (2)
            to  node[below] {$y_9$} (3)
  			to node[left]{$y_7$}(1);
  	\draw (1) to node[left]   {$y_2$} (0)
  			 to	 node [left]   {$y_3$} (2);
  \draw (1) to node[below left]   {$z_1$} (i)
  			 to	 node [below]   {$z_2$} (2);
   	\draw (0) to node[below]   {$y_4$} (3)
  			 to	 node [above]   {$z_3$} (i)
  			 to	 node [above]   {$z_0$} (0);
    \draw (1) to node[above]   {$y_8$} (4)
  			 to	 node [left]   {$y_{10}$} (2);
      \draw (3) to node[right]   {$y_{11}$} (4)
  			 to	 node [above right]   {$y_5$} (0);
  	    \draw (i) to node[above right]   {$z_4$} (4);
\end{tikzpicture}
}
  \caption{The Augmented Weighted Undirected Complete Graph $G$}
\label{graph:G}
\end{figure}

\begin{remark}
The above importantly suggests that if a subclass of TDMs is such that the equation $\mathbf{C}_d\mathbf{x}=\mathbf{p}_d$, $\mathbf{x}\geq \mathbf{0}$ of Theorem \ref{FABC} {\it ii.} can be reduced such that the separation problem can be phrased in terms of a combinatorial optimization problem that is solvable in polynomial time, then the realization problem for this subclass of TDMs can be solved in polynomial time. For example, the max-cut problem for certain sub-classes of undirected real weighted graphs have been shown to be solvable in polynomial time. Some of these include \citet{Barahona1983} which considers classes of graphs not contractible to, $K_5$,  the complete graph with $5$ vertices; a better algorithm for planar graphs (which are not contractible to $K_5$ by Wagner's theorem) is given in \citet{Shih1990}; for series-parallel graphs see \citet{Xu}. Identifying subclasses along the above suggested lines is beyond the scope of our current work.
\end{remark}

\section{Reducing Complexity: Symmetry and Sparsity}
In this section, we show that if a given matrix exhibits  invariance with respect to permutations of indices, then there is a resultant decrease in the complexity of the realization problem for TDMs and BCMs. 
Also, there is a significant reduction in complexity in the presence of zero elements in the TDM/BCM parametrization. 

In the following, $\mathcal{S}_d$ denotes the symmetric group on the set $\{1,\ldots,d\}$ - that is the set of all permutations of $\{1,\ldots,d\}$ with the composition of maps as the product operation and the neutral map serving as the identity. Note that the cardinality of $\mathcal{S}_d$ is $d!$. Let $\mathbf{P}_\sigma$ denote the permutation matrix associated with the permutation $\sigma$. To be specific, the matrix $\mathbf{P}_\sigma$ is in $\{0,1\}^{d^2}$ with its $ij$-th element being $1$ \iff $i=\sigma(j)$. So the identity permutation corresponds to the identity matrix, $\mathbf{P}_\sigma \mathbf{A}$ (resp., $\mathbf{A}\mathbf{P}_\sigma^{\mathsf{T}}$) results in a matrix which is a row (resp., column) permutation of $\mathbf{A}$.

For a given matrix $\mathbf{A}$, let $\mathcal{G}_\mathbf{A} \subseteq \mathcal{S}_d$ be the set of all permutation of the indices under which $\mathbf{A}$ is invariant. It is easily verified that $\mathbf{A}$ is invariant under $\sigma$-permutation of its indices \iff $\mathbf{P}_\sigma \mathbf{A} \mathbf{P}_\sigma ^ \mathsf{T}=\mathbf{A}$. We note that $\mathcal{G}_\mathbf{B}$ is a subgroup of $\mathcal{S}_d$, which follows from the fact that $\mathbf{P}_\sigma \cdot \mathbf{P}_\delta$ is the permutation matrix corresponding to ${\sigma \cdot \delta = \sigma\circ \delta}$, for $\sigma,\delta \in\mathcal{S}_d$. 

We note that matrices formulated in Example \ref{ex:two Sector} below are all invariant with respect to the subgroup  $\mathcal{S}_{d_1+d_2}$ consisting of permutations which are composition of two permutations, one which permutes indexes ${\{1,\ldots,d_1\}}$ among themselves and another which permutes indexes $\{d_1+1,\ldots,d_1+d_2\}$ among themselves. Note that the cardinality of this subgroup is $d_1! d_2!$. Similarly, the sub-group corresponding to the matrices formulated in Example {\ref{ex:equi-correlation}} is the whole of $\mathcal{S}_d$ with cardinality $d!$. 

The following lemma states that a BCM $\mathbf{B}$ necessarily corresponds to a Bernoulli random vector that inherits the symmetry in $\mathbf{B}$ as represented by $\mathcal{G}_\mathbf{B}$. This is key to reducing the complexity of the realization problem. 

\begin{lemma}\label{Symm}
A matrix $\mathbf{B}_{d \times d}$ is in $\calB_d$ if and only if there exists a random vector $\mathbf{X}$ taking values in $\{0,1\}^d$ with  $\mathbf{B}=\E{\mathbf{X}\mathbf{X}^{\mathsf{T}}}$, and satisfying $\mathbf{X}\eqd \mathbf{P}_\sigma \mathbf{X}$ for all $\sigma \in \mathcal{G}_\mathbf{B}$.
\end{lemma}

\begin{proof}
Note that the sufficiency part follows trivially from the definition of $\calB_d$. For the necessity part, for a given $\mathbf{B}_{d \times d}\in \calB_d$, let $\mathbf{X}$ be a random vector taking values in  $\{0,1\}^d$ satisfying $\mathbf{B}=\E{\mathbf{X}\mathbf{X}^{\mathsf{T}}}$. Let $\sigma_U$ be independent of $\mathbf{X}$ and be distributed uniformly on the subgroup $\mathcal{G}_\mathbf{B}$. Since $\delta \cdot \sigma_U \eqd \sigma_U$ for all $\delta \in \mathcal{G}_\mathbf{B}$, we have $\mathbf{Y}:=\mathbf{P}_{\sigma_U}\mathbf{X}$ satisfies $\mathbf{P}_\delta \mathbf{Y} \eqd \mathbf{Y}$ for all $\delta \in \mathcal{G}_\mathbf{B}$. The proof is completed by observing that
\[
\mathbf{B}=\E{\mathbf{P}_{\sigma_U} \mathbf{B} \mathbf{P}_{\sigma_U}^\mathsf{T}}= 
\E{\mathbf{P}_{\sigma_U} \mathbf{X}\mathbf{X}^{\mathsf{T}} \mathbf{P}_{\sigma_U}^\mathsf{T}} = \E{\mathbf{Y}\mathbf{Y}^{\mathsf{T}}}.
\] 
\end{proof}

Let $\mathbf{B}$ be a $d\times d$ matrix. We say that two functions $g$ and $h$ on $\{1,\ldots,d\}$ are $\mathcal{G}_\mathbf{B}$ equivalent if $g=h\circ \sigma$ for some $\sigma\in \mathcal{G}_\mathbf{B}$. Note that the space of functions from $\{1,\ldots,d\}$ to $\{0,1\}$ is isomorphic to $\{0,1\}^d$. Since $\mathcal{G}_\mathbf{B}$ is a group with composition as the group operation, it follows that this defines an equivalence relation on $\{0,1\}^d$, which we denote by $\gbequiv$. In the following we will denote the cardinality of the quotient set $\{0,1\}^d/\gbequiv$ by $N(\mathbf{B})$. Let $\pi_\mathbf{B}$ denote the projection from $\{0,1\}^d$ to $\{0,1\}^d/\gbequiv$, and the $2^d \times N(\mathbf{B})$ matrix $\Pi_\mathbf{B}$ be such that its $\mathbf{i}$-th row, $\mathbf{i}\in \{0,1\}^d$, in $\myord$ ordering has a one in its column corresponding to $\pi_\mathbf{B}(\mathbf{i})$ with the rest of the elements being zero. The following theorem makes precise the reduction in complexity of the realization problem that arises from symmetry. 

\begin{theorem} \label{reduce_thm}
For any matrix  $\mathbf{B}_{d \times d}$, the following are equivalent:
\begin{enumerate}
\item $\mathbf{B}\in \calB_d$ 
\item There exists a vector $\mathbf{x}$ in $[0,\infty)^{N(\mathbf{B})}$ such that $\mathbf{C}_d\Pi_\mathbf{B}\mathbf{x}=\mathbf{p}_d(\mathbf{B})$.
\end{enumerate}
\end{theorem}

\begin{proof}
From Theorem \ref{FABC} we have {\it ii.} implies {\it i}. For {\it i.} implies {\it ii.}, by Lemma \ref{Symm}, we have a random vector 
$\mathbf{X}$ taking values in $\{0,1\}^d$, satisfying  $\mathbf{B}=\E{\mathbf{X}\mathbf{X}^{\mathsf{T}}}$, and $\mathbf{X}\eqd \mathbf{P}_\sigma \mathbf{X}$ for all $\sigma \in \mathcal{G}_\mathbf{B}$. This implies that the vector $\mathbf{q}_d$, corresponding to such a random vector $\mathbf{X}$, can be chosen using the above discussion to satisfy
$\mathbf{q}_d=\Pi_\mathbf{B}\mathbf{x}$, for $\mathbf{x}$ in $[0,\infty)^{N(\mathbf{B})}$. Since $\mathbf{C}_d\mathbf{q}_d=\mathbf{p}_d(\mathbf{B})$, we have $\mathbf{C}_d\Pi_\mathbf{B}\mathbf{x}=\mathbf{p}_d(\mathbf{B})$. 
\end{proof}

\begin{remark}
In the above theorem, we focused on reduction in the number of variables in the LP formulation of the realization problem, as its size equals $2^d$ whereas the number of constraints is polynomial in $d$. Nevertheless, examples below exhibit reduction in the number of constraints as well. While $\mathcal{G}_\mathbf{B}$ having at least two elements will reduce the number of variables, the reduction is much more fruitful in parametrizations which yield $N(\mathbf{B})$ with polynomial dependence on $d$; in the latter case the realization problem restricted to the parametric class becomes polynomial in complexity.
\end{remark}

\begin{remark}
Towards describing an algorithm to compute $N(\mathbf{B})$, we begin by considering $d\times d$ matrices $\mathbf{B}$ with a constant main diagonal, and $k$ distinct off-diagonal values. Such matrices can naturally be mapped to an edge labeled graph with $d$ vertices and $k-1$ distinct edge labels. We now observe that this latter constructed graph's automorphism group, a permutation group on $\{1,\ldots,d\}$, equals $\calG_{\mathbf{B}}$. The problem of finding the generators of this automorphism group is polynomially equivalent to the graph isomorphism problem, as shown in \citet{Mathon1979}. Moreover, while the exact complexity of the graph isomorphism problem has not yet been established, it has been recently shown to admit a quasipolynomial time algorithm, see \cite{Babai2016}. On the practical side there exists many competitive algorithms, and \citet{Mckay2014} is a recent survey of such algorithms and software libraries. We note that $N(\mathbf{B})$ equals the number of orbits (of $\calG_{\mathbf{B}}$ acting on $\{0,1\}^d$), and can be computed by using P\'olya's enumeration theorem. Finally, in the case of matrices $\mathbf{B}$ with non-constant main diagonal with $l$ distinct values, we can either introduce vertex labeling with $l$ vertex labels or create self loops with $l$ additional edge labels. The latter approach is particularly conducive to using SageMath (see \citet{sagemath}) to compute $N(\mathbf{B})$, and we demonstrate this in Figure \ref{fig:sagemath} of the Appendix.
\end{remark}

The equi-correlation parametrization of Example \ref{ex:equi-correlation} exhibits the largest possible $\calG_\mathbf{B}$, and hence the lowest possible value of $d+1$ for $N(\mathbf{B})$. This reduction is used below to find necessary and sufficient conditions on the parameters, {\it i.e.} a description of the constraint polytope,  for the matrix to be a BCM.\\

\noindent \emph{Example \ref{ex:equi-correlation} (Continued)}
In the following, we derive $\beta_l(\cdot)$, which will specify that the conditions $\alpha\in[0,1]$ and  $\beta \geq \beta_l(\alpha)$ are together both necessary and sufficient  for the matrix of the form \eqref{equi-corr-mat} to be a BCM. From Corollary \ref{BCprobcor}, we know that if $\beta \geq \beta_l(\alpha)$ for some $\alpha\in[0,1]$, then there exists events $A_i$ such that $\Pr{A_i}=\alpha$, for $i=1,\ldots,d$, and $\Pr{A_i \cap A_j}=\beta$, $1 \leq i < j \leq d$. By Lemma \ref{Symm} above, we can assume that for any $k=0,\ldots,d$, $q_{i_1,...,i_d}$ defined in \eqref{qdef1} satisfies 
\[
q_{i_1,...,i_d}=x_k, \quad \hbox{for } (i_1,\ldots,i_d)\in\{0,1\}^d,\; \text{ with } \; k=\sum_{j=1}^d i_j
\]  
for some non-negative reals $x_k$, $k=0,\ldots,d$. Moreover, $\mathbf{C}_d\mathbf{q}_d=\mathbf{p}_d$ reduces to three equations. The first equation is easily seen to reduce to $\sum_{k=0}^d \binom{d}{k}  x_k =1$. The next $d$ equations all reduce to $\sum_{k=1}^d \binom{d-1}{k-1} x_k =\alpha$, with $\binom{d-1}{k-1}$ representing the number of 
binary tuples $(i_1,\ldots,i_d)$ with $\sum_{j=1}^d i_j=k$ and $i_m=1$ for some fixed $m$. The remaining $\binom{d}{2}$ equations  all reduce to $\sum_{k=2}^d \binom{d-2}{k-2} x_k =\beta$ with $\binom{d-2}{k-2}$ representing the number of binary tuples $(i_1,\ldots,i_d)$ with $\sum_{j=1}^d i_j=k$ and $i_m=1=i_n$ for some fixed $m\neq n$. From the previous discussion it follows that $\beta_l(\alpha)$, is the optimal value of the following linear programming problem:
    	\begin{eqnarray*}
			&\text{minimize}& \sum_{k=2}^d  \begin{pmatrix}
				d-2 \\
				k-2
			\end{pmatrix}  x_k; \\
			&\text{subject to} &   \sum_{k=1}^d  
            \begin{pmatrix}
				d-1 \\
				k-1
			\end{pmatrix}  x_k =\alpha; \qquad
            \sum_{k=0}^d  
            \begin{pmatrix}
				d \\
				k
			\end{pmatrix}  x_k =1;\qquad x_k \geq 0,\; k=0,\ldots,d.
		\end{eqnarray*}
Using duality, we note that $\beta_l(\alpha)$ is the optimal value of the following linear programming problem as well.  
    \begin{eqnarray*}
		&\text{maximize}& \qquad \alpha y_1+y_2;  \\
		& & \\
		&\text{subject to} & \qquad y_2 \leq 0; \qquad y_1+dy_2 \leq 0; \\
		& & \\
		& & \qquad
        \begin{pmatrix}
			d-1\\
			k-1
		\end{pmatrix}  y_1+ 
        \begin{pmatrix}
			d\\
			k
		\end{pmatrix}  y_2 \leq 
        \begin{pmatrix}
			d-2\\
			k-2
		\end{pmatrix},\; \text{for } k=2,...,d.\nonumber
	\end{eqnarray*}
This is so as $(x_0,x_1,\ldots,x_{d-1},x_d)=(1-\alpha,0,\ldots,0,\alpha)$ is a feasible solution in the first problem and  $(y_1,y_2)=(0,0)$ is a  feasible solution in the second. In the following, we solve the latter linear programming problem. 
    
It can be shown that $\left(\frac{2k}{d-1},-\frac{k(k+1)}{d(d-1)} \right)$, for $k=0,\ldots,d$, form the vertices of the polyhedral solution space. So the optimal value is attained at one of these vertices, and hence we maximize $\alpha y_1+y_2$ over these $d$ vertices. Let $\alpha d=j+\varepsilon$, for $j$ in  $\{0,\ldots,d\}$ and $\varepsilon$ in $[0,1)$. Then
\begin{multline*}
\alpha y_1+y_2=\frac{(j+\varepsilon)y_1+dy_2}{d}=\frac{-k^2+2kj+2k\varepsilon-k}{d(d-1)}\\=\frac{-(k-j-\varepsilon+\frac{1}{2})^2+(j+\varepsilon-\frac{1}{2})^2}{d(d-1)}.
\end{multline*}
It follows from above that the maximum is attained when $j=k$, in which case $\alpha y_1+y_2$ equals 
\[
\frac{j^2+2j\varepsilon-j}{d(d-1)}.
\]
In other words, $\beta_l(\alpha)$ is given by
\begin{eqnarray}
		 \beta_l(\alpha)=\frac{ \left(2\alpha d- \lfloor  \alpha d \rfloor-1\right) \lfloor  \alpha d \rfloor}{d(d-1)}. \label{LB}
	\end{eqnarray}

\begin{remark}
In Problem Class 1 of \citet{Krause2018} they consider a sub-class of the above with $\beta=\kappa\alpha^2$, with the parametrization $(\alpha,\kappa)\in [0,1]^2$. Note that, using the above analysis, we have for each fixed $\alpha$, the lower bound on $\kappa$ such that $(\alpha,\kappa\alpha^2)$ belongs to the above constraint polytope is given by 
\[
 \frac{ \left(2\alpha d- \lfloor  \alpha d \rfloor-1\right) \lfloor  \alpha d \rfloor}{\alpha^2 d(d-1)}.
\]
In Figure 5 of their paper, they have graphed this region of the parameter space, supported by numerical results, for a few dimensions.  
\end{remark}

\begin{remark}
For dimension $d$, the number of facets of the polytope in $[0,1]^2$ of values $(\alpha,\beta)$ resulting in a equi-correlation BCM is $d+1$. Figure \ref{equi_3} (a) and (b) show the polytopes for the cases when $d=2$ and $d=3$, respectively. Also, when $d \uparrow \infty$, $\beta_l(\alpha) \uparrow \alpha^2$ for any $\alpha\in[0,1]$, with the lower boundary of the limiting convex set corresponding to the independent case (see bold curve in Figure \ref{equi_3} (c)).

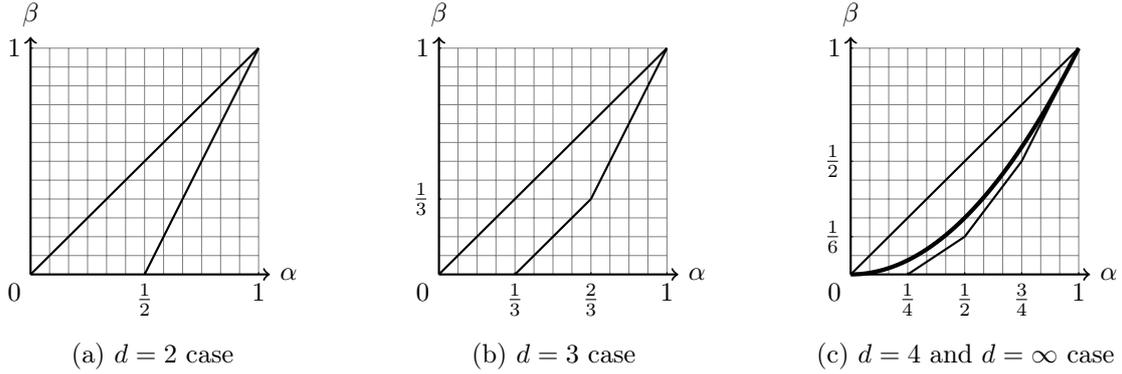
\begin{figure}
\captionsetup[subfigure]{font=footnotesize}
\centering
\subcaptionbox{$d=2$ case}[.33\textwidth]{%
\begin{tikzpicture}[scale=3]\footnotesize
 \pgfmathsetmacro{\xone}{0}
 \pgfmathsetmacro{\xtwo}{1.05}
 \pgfmathsetmacro{\yone}{0}
 \pgfmathsetmacro{\ytwo}{1.05}
\begin{scope}<+->;
  \draw[step=.08333cm,gray,very thin] (\xone,\yone) grid (1,1);
  \foreach \x/\xtext in { .5/{\frac{1}{2}}, 1/1}
  \draw[black,xshift=\x cm] (0,.01) -- (0,0) node[below] {$\xtext$};
  \foreach \y/\ytext in { 1/1}
    \draw[black, yshift=\y cm] (.01,0) -- (0,0)   node[left] {$\ytext$};
 \draw[black] (0,0) node[anchor=north east] {$0$};
  \draw[thick,->] (\xone, 0) -- (\xtwo, 0) node[right] {$\alpha$};
  \draw[thick,->] (0, \yone) -- (0, \ytwo) node[above] {$\beta$};
   \draw[thick] (0, 0) -- (1, 1);
      \draw[thick] (.5, 0) -- (1,1) ;
\end{scope}
\end{tikzpicture}
}%
\subcaptionbox{$d=3$ case}[.33\textwidth]{
\begin{tikzpicture}[scale=3]\footnotesize
 \pgfmathsetmacro{\xone}{0}
 \pgfmathsetmacro{\xtwo}{1.05}
 \pgfmathsetmacro{\yone}{0}
 \pgfmathsetmacro{\ytwo}{1.05}
\begin{scope}<+->;
  \draw[step=.08333cm,gray,very thin] (\xone,\yone) grid (1,1);
  \foreach \x/\xtext in { .33333/{\frac{1}{3}} ,.66667/{\frac{2}{3}}, 1/1}
  \draw[black,xshift=\x cm] (0,.01) -- (0,0) node[below] {$\xtext$};
  \foreach \y/\ytext in { .33333/\frac{1}{3}, 1/1}
    \draw[black, yshift=\y cm] (.01,0) -- (0,0)   node[left] {$\ytext$};
 \draw[black] (0,0) node[anchor=north east] {$0$};
  \draw[thick,->] (\xone, 0) -- (\xtwo, 0) node[right] {$\alpha$};
  \draw[thick,->] (0, \yone) -- (0, \ytwo) node[above] {$\beta$};
   \draw[thick] (0, 0) -- (1, 1);
   \draw[thick] (.33333, 0) -- (.66667, .33333) ;
     \draw[thick] (.66667, .33333) -- (1, 1) ;
\end{scope}
\end{tikzpicture}}
\subcaptionbox{$d=4$ and $d=\infty$ case}[.33\textwidth]{
\begin{tikzpicture}[scale=3]\footnotesize
 \pgfmathsetmacro{\xone}{0}
 \pgfmathsetmacro{\xtwo}{1.05}
 \pgfmathsetmacro{\yone}{0}
 \pgfmathsetmacro{\ytwo}{1.05}
\begin{scope}<+->;
  \draw[step=.08333cm,gray,very thin] (\xone,\yone) grid (1,1);
  \foreach \x/\xtext in {.25/\frac{1}{4}, .5/\frac{1}{2}, .75/\frac{3}{4}, 1/1}
  \draw[black,xshift=\x cm] (0,.01) -- (0,0) node[below] {$\xtext$};
  \foreach \y/\ytext in { .1666667/\frac{1}{6}, .5/\frac{1}{2},  1/1}
    \draw[black, yshift=\y cm] (.01,0) -- (0,0)   node[left] {$\ytext$};
 \draw[black] (0,0) node[anchor=north east] {$0$};
  \draw[thick,->] (\xone, 0) -- (\xtwo, 0) node[right] {$\alpha$};
  \draw[thick,->] (0, \yone) -- (0, \ytwo) node[above] {$\beta$};
   \draw[thick] (0, 0) -- (1, 1);
   \draw[thick] (.25, 0) -- (.5, .166667) ;
     \draw[thick] (.5, .16667) -- (.75, .5) ;
     \draw[thick]  (.75, .5) -- (1,1) ;
     \draw[ultra thick,domain=0:1,smooth] plot (\x,{(\x)^2});
\end{scope}
\end{tikzpicture}}
\caption{Polytope of TDM Generating Values of $(\alpha, \beta)$} \label{equi_3}
\end{figure}
\end{remark}

\begin{remark}
Note that for a $d$-dimensional random vector, the variance of the sum of its elements is nonnegative. This fact can be used to derive a lower bound, say $b_l(\alpha)$ on $\beta$, for each $\alpha\in[0,1]$; this is given by  
\begin{eqnarray}
	b_l(\alpha) = \frac{\alpha^2d-\alpha}{d-1}. \label{LB_b}
\end{eqnarray}
Comparing $\beta_l$ and $b_l$, for $\alpha d=k+\varepsilon$, where $k=\{0,\ldots,d\}$ and $\varepsilon \in [0,1)$, we get
\[
b_l(\alpha)=\beta_l(\alpha)+ \frac{\varepsilon(1-\varepsilon)}{d(d-1)}.
\]
Thus, when $\varepsilon=0$ (or equivalently $\alpha d$ is an integer) the above two lower bounds coincide, but not otherwise. 
\end{remark}

So far we have focused on the reduction in the complexity of the realization problem gained by the inherent symmetry in the parametric class. But some parameterizations of practical interest lack much symmetry. One such parametric class which results in a small $\calG_\mathbf{B}$ is the class of Toeplitz matrices (see \eqref{Toep_matrix}). To see this, we note that a matrix $\mathbf{A}$ is called centro-symmetric if its elements $a_{ij}$ satisfy 
\[
a_{i,j} = a_{n+1-i,n+1-j}, \quad 1\leq i,j \leq n.
\]
Equivalently, a centro-symmetric matrix is one that is invariant to permuting its rows and columns by the permutation $\sigma$ given by 
\[
\sigma=
\begin{pmatrix}
1 & 2 & \cdots & n-1 & n \\
n& n-1 & \cdots & 2 & 1
\end{pmatrix}.
\]
All Toeplitz matrices are centro-symmetric and $\calG_\mathbf{B}$ consists of only the identity and the above permutation.  

Nevertheless, reduction in complexity can also be achieved if there is a significant presence of zero elements in the parametrization, {\it i.e.} a sparse parametrization. As one such example, we consider below the parametric class of matrices corresponding to $m$-dependent stationary sequences. 

In an $m$-dependent stochastic process indexed by $\mathbb{Z}$, the observations at times more than $m$ apart are independent. This implies that the TDMs corresponding to such a stochastic process has elements equal to zero that are not on any of the $k$-diagonals, for $k=0,\ldots,m$. In addition, as mentioned earlier, if the process is stationary then the TDM corresponding to such a stochastic process is symmetric Toeplitz. This gives rise to a $m$-dimensional parametrization of such TDMs; a natural question is which of such symmetric Toeplitz matrices is a valid $d$-dimensional TDM. The proposition below shows that complexity of the $d$-dimensional realization problem restricted to such symmetric Toeplitz matrices is at most linear in $d$. 

\begin{proposition}\label{prop:m-dependence}
Consider the class, $\mathcal{T}_{d:m}$, of $d$-dimensional matrices which are symmetric Toeplitz with elements equal to zero that are not on any of the $k$-diagonals for $k=0,\ldots,m$. The linear programming formulation of the TDM realization problem for this class, and for a fixed $m$ ({\it i.e.} free of $d$), has at most linear in $d$ number of constraints and variables.  
\end{proposition}
\begin{proof}
From Theorem \ref{FABC}, we know that for a given $d$-dimensional matrix $\mathbf{T}$, the verification that it represent a TDM is equivalent to verifying the existence of $\mathbf{x}\geq 0$ such that $\mathbf{C}_d\mathbf{x}=\mathbf{p}_d$, where $\mathbf{p}_d$ is constructed using the elements of $(1/d)\mathbf{T}$. In this formulation, $\mathbf{x}$ has length $2^d$ and the number of constraints is $d(d+1)/2+1$. We argue below that $\mathbf{x}$ lies in a subspace of dimension which is linear in $d$, as well as the number of constraints can be reduced to be linear in $d$.  
 
Note that in the case that $\mathbf{T}$ in $\mathcal{T}_{d:m}$ is a TDM, $(1/d)\mathbf{T}$ is a BCM according Theorem \ref{BCTD}. Appealing to the probabilistic viewpoint, we consider for $(1/d)\mathbf{T}$ the corresponding equation $\mathbf{C}_d\mathbf{q}_d=\mathbf{p}_d$. Let $A_i$, $i=1,\ldots,d$, be $d$ events on some probability space corresponding to $(1/d)\mathbf{T}$ as specified by \eqref{BCprob} and underlying \eqref{qdef1}, \eqref{qdef2}, and \eqref{defn:pd}. Since scaled TDMs will have non-zero main diagonal elements, the number of non-zero elements of $\mathbf{p}_d$ is at most one more that the cardinality of the set 
\[
\left\{(i,j): \vert i-j\vert \leq m,\; 1\leq i< j \leq d \right\},
\]
which equals $(1/2)(2d-m)(m+1)+1$. Notice that a zero element of $\mathbf{p}_d$ corresponding to $\Pr{A_j\cap A_k}$, for $1\leq j<k\leq d$, forces the elements of $\mathbf{q}_d$ corresponding to indexes in 
\[
\left\{\textbf{i}\in\{0,1\}^d \big\vert i_j=i_k=1 \right\}
\]
to be zero. From this observation it follows that for $\mathbf{p}_d$ corresponding to matrices in $\mathcal{T}_{d:m}$, the non-zero elements of $\mathbf{q}_d$ have indexes in 
\begin{equation}\label{theset}
\left\{\textbf{i}\in\{0,1\}^d \big\vert i_j,i_k=1, j\neq k \implies \vert j-k\vert \leq m \right\},
\end{equation}
with the cardinality of this set being $(d-m+1)2^{m}$. Let $\mathbf{C}'_d$ equal the matrix $\mathbf{C}_d$ without its rows corresponding to zero elements of $\mathbf{p}_d$ and containing only its columns  corresponding to indexes in the set given in \eqref{theset}. Also, let $\mathbf{p}'_d$ equal $\mathbf{p}_d$ without its zero elements. The above discussion implies that for matrices in $\mathcal{T}_{d:m}$,  verifying the existence of $\mathbf{x}\geq \mathbf{0}$ such that $\mathbf{C}_d\mathbf{x}=\mathbf{p}_d$ is equivalent to verifying the existence of $\mathbf{x}\geq \mathbf{0}$ such that $\mathbf{C}'_d\mathbf{x}=\mathbf{p}'_d$. This latter problem is equivalent to a linear programming problem with linear in $d$ number of variables and constraints. 
\end{proof}

\begin{remark} \label{prop:m-dependence:rem}
In the above proof, we note that we can further reduce the subspace containing $\mathbf{q}_d$ by using the fact that all Toeplitz matrices are centro-symmetric. It can be checked that this subspace has dimension equal to 
\[
\begin{cases}
(d-m+1)2^{m-1}+ 2^{\lceil m/2 \rceil -1}, &\hbox{ for $d$ even};\\
(d-m+1)2^{m-1}+ 2^{\lceil m/2 \rceil}, &\hbox{ for $d$ odd}.\\
\end{cases}
\]
It is worth mention that for general Toeplitz matrices, the reduction is only to a subspace of dimension as large as $2^{d-1}+2^{\lceil d/2 \rceil -1}$.
\end{remark}

In Proposition \ref{Toep}, we gave a sufficient but not necessary condition for a symmetric $d$-dimensional Toeplitz matrix to be a TDM. In the following proposition, we consider a $d$-dimensional TDM that corresponds to $d$ successive observations from a $2$-dependent stationary process; such TDMs can be parametrized as 
\begin{equation} \label{2dep}
\begin{bmatrix}
1 & \alpha & \beta & 0 & \cdots & 0 \\
\alpha & 1 & \alpha & \ddots & \ddots  & \vdots \\
\beta & \alpha & \ddots & \ddots & \ddots& 0 \\
0 & \ddots &\ddots & \ddots & \alpha &\beta \\
\vdots & \ddots  &\ddots & \alpha & 1 & \alpha\\
0 & \cdots & 0 &\beta & \alpha &1
\end{bmatrix}.
\end{equation}
Unlike the case of a general Toeplitz matrix, we are able to provide a necessary and sufficient condition on $\alpha$ and $\beta$ in order for \eqref{2dep} to be a TDM.  

\begin{proposition} \label{MA2}
For any dimension $d \geq 6$, the two-dependence matrix of the form \eqref{2dep} is in $\mathcal{T}_d$ if and only if $\alpha$ and $\beta$ satisfy the following constraints:
\begin{eqnarray}\label{cond2dep}
\begin{cases}
\alpha, \beta \geq 0; \\
\alpha+4\beta \leq 2;  \\
 2\alpha-\beta \leq 1.
\end{cases}
\end{eqnarray}
\end{proposition}
\begin{proof}
{\it Necessity Part:} Since the form of the matrix in \eqref{2dep} is such that the $6$ dimensional principal sub-matrix is again of the same form, we note that it suffices to show that the constraints in \eqref{cond2dep} are necessary for the $d=6$ case. 

We denote a general $d$ dimensional matrix of the form in \eqref{2dep} by  $\mathbf{T}_d(\alpha,\beta)$. Let\\ $\mathbf{T}_6(\alpha,\beta) \in \mathcal{T}_6$. By Corollary \ref{BCprobcor}, we know that $\mathbf{B}_6:=\mathbf{T}_6/6 \in \mathcal{B}_6$. Let the sets $A_i$, for $i=1,\ldots,6$, be those corresponding to $\mathbf{B}_6$ and as specified in \eqref{BCprob}. Since $\mathbf{B}_6$ is centro-symmetric, 
the permutation 
\[
\begin{pmatrix}
1 & 2 & 3 &4 & 5 & 6 \\
6& 5 & 4 & 3 & 2 & 1
\end{pmatrix}
\]
belongs to $\mathcal{G}_{\mathbf{B}_6}$. Thus by Theorem \ref{reduce_thm}, without loss of generality, we could assume that for some $\gamma_1,\gamma_2$,  
\begin{multline*}
\Pr{A_1 \cap A_2 \cap A_3}=\Pr{A_4 \cap A_5 \cap A_6}=\gamma_1, \\ \hbox{ and } \Pr{A_2 \cap A_3 \cap A_4}=\Pr{A_3 \cap A_4 \cap A_5}=\gamma_2.    
\end{multline*}
Denoting $\alpha':=\alpha/6$ and $\beta':=\beta/6$, this results in a Venn diagram for the sets $A_i$, for $i=1,\ldots,6$, as given in Figure \ref{2depnecVenn}. 

	\begin{figure}
\centering
\begin{tikzpicture}[thick,scale=1, every node/.style={scale=0.6}]
  \tikzset{venn circle/.style={draw,circle,minimum width=4cm,fill=#1,opacity=0.5}}
  \draw[very thick] (6,-3.5) -- (-3,-3.5);
  \draw[very thick] (-3,-3.5) -- (-3,3);
  \draw[very thick] (6,3) -- (-3,3);
  \draw[very thick] (6,3) -- (6,-3.5);
  \draw (-1,1) circle (1.5cm);
  \draw (1,1) circle (1.5cm);
  \draw (0,-.7) circle (1.5cm);
  \draw (2,-.7) circle (1.5cm);
  \draw (3,1) circle (1.5cm);
  \draw (4,-.7) circle (1.5cm);
   \node at (-1.2,1.7) {\Large  $A_1$}; 
   \node at (1,1.7) {\Large   $A_3$}; 
   \node at (3.1,1.7) {\Large   $A_5$}; 
   \node at (-.2,-1.1) {\Large   $A_2$}; 
   \node at (2,-1.1) {\Large   $A_4$}; 
   \node at (4.3,-1.1) {\Large   $A_6$}; 
     \node at (0,0.4) {\large $\gamma_1$}; 
     \node at (1,-.1) {\large $\gamma_2$}; 
     \node at (2,.4) {\large $\gamma_2$}; 
     \node at (3,-.1) {\large $\gamma_1$}; 
    \node at (-.7,0) { \large$\alpha'-\gamma_1$}; 
    \node at (0,1.2) {\large $\beta'-\gamma_1$}; 
     \node at (2,1.2) {\large  $\beta'-\gamma_2$}; 
     \node at (1,-.8) {\large  $\beta'-\gamma_2$}; 
     \node at (3,-.8) {\large $\beta'-\gamma_1$}; 
      \node at (3.8,0.4) { \large  $\alpha'-\gamma_1$}; 
	\node[rotate=60] at (.5,.2) { $\alpha'-\gamma_1-\gamma_2$}; 
     \node[rotate=-45] at (1.5,.2) {$\alpha'-2\gamma_2$}; 
     \node[rotate=60] at (2.5,.2) {$\alpha'-\gamma_1-\gamma_2$}; 
  \node[below] at (3,-2.5 ) {\Large $A_1^c \cap \cdots \cap A_6^c$}; 
\end{tikzpicture}
\caption{Venn Diagram for Sets $A_i$, $i=1,\ldots, 6$}
\label{2depnecVenn}
\end{figure}
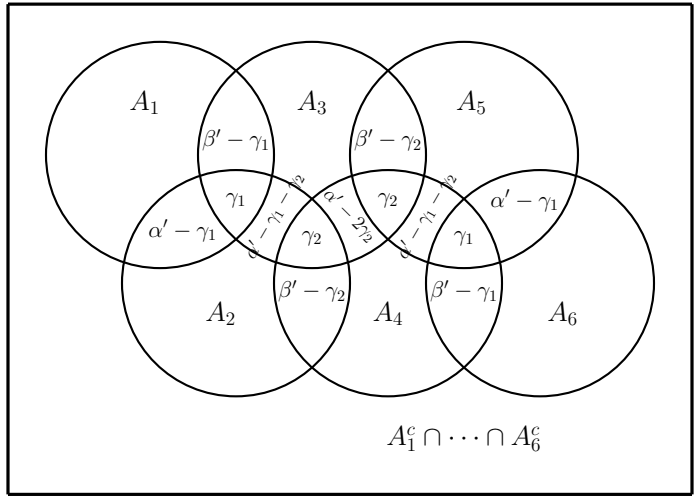
Imposing the non-negativity constraint on the probability of each set in the partition generated by the sets $A_i$, for $i=1,\ldots,6$, and the constraint that $\Pr{A_i}=1/6$, for $i=1,\ldots,6$, we have the following inequalities:
\begin{eqnarray}
0 \leq \gamma_1,\gamma_2 \leq \beta'; \label{c1} \\ 
\alpha'-\gamma_1-\gamma_2 \geq 0; \label{c2} \\
\alpha'-2\gamma_2 \geq 0 ;\label{c3} \\ 
2\alpha'+2\beta'-\gamma_1-2\gamma_2 \leq 1/6. \label{c4}
\end{eqnarray}
It is easy to see that (\ref{c1}) and (\ref{c4}) imply $2\alpha-\beta \leq 1$, and ${\alpha + 4\beta \leq 2}$ is implied by (\ref{c2}), (\ref{c3}) and (\ref{c4}). This completes the proof of the necessity part. 

{\it Sufficiency Part:} In the following, we will show that $\mathbf{T}_d(\alpha,\beta)$, for $\alpha,\beta$ satisfying in \eqref{cond2dep}, belongs to $\mathcal{T}_d$, for any $d\geq 6$. By looking at the principal sub-matrices, this implies sufficiency of the conditions in \eqref{cond2dep}, for all $d\geq 3$. By Corollary \ref{BCprobcor}, we know that $\mathbf{B}_d:=\mathbf{T}_d/d \in \mathcal{B}_d$. Let the sets $A_i$, for $i=1,\ldots,d$, be those corresponding to $\mathbf{B}_d$, and as specified in \eqref{BCprob}. 

Let $d \geq 6$. We define, 
\begin{equation} \label{defnkappa}
\kappa:=\alpha/d; \quad \mu:=\beta/d; \quad \hbox{and } \nu=\min\{\kappa/2,\mu\}.
\end{equation}
Let $A_i$, for $i=1,\ldots,d$, be events on the probability space $([0,1],\mathcal{L},\lambda)$ such that they agree with the Venn diagram given in Figure \ref{2depsuffVenn} of the Appendix. Note that in our construction, sets $A_i$ and $A_j$ for $j-i>2$ are disjoint. It can be argued that there are $4d-4$ sets in the partition which can be mapped to $\{0,1\}^d$ by using the representation $\cap_{i=1}^d A_i^{j_i}$, with $\mathbf{j}:=(j_1,\ldots,j_d)\in\{0,1\}^d$. Among the $4d-4$ binary vectors, $4d-8$ are of the form 
\begin{equation}\label{binvec1}
(\underbrace{0,\ldots,0}_{m},1,p,q,\underbrace{0,\ldots,0}_{n}), \quad m,n\geq 0;\; m+n=d-3; (p,q)\in\{0,1\}^2,
\end{equation}
and the rest four are of the form 
\begin{equation}\label{binvec2}
(\underbrace{0,\ldots,0}_{d-2},p,q), \quad (p,q)\in\{0,1\}^2.
\end{equation}
Since the sets in the partition generated by $A_i$, $i=1,\ldots,d$, can be assigned to be disjoint intervals, for well definition all we need to check is that the probability of these sets in the partition are non-negative and sum to one. The probabilities of the $4d-5$ sets, excluding the set $A_1^\mathsf{c}\cap\cdots\cap A_d^\mathsf{c}$, are as given below:
\begin{equation}
\Pr{\cap_{i=1}^d A_i^{j_i}}=
\begin{cases}
\xi_1, & \mathbf{j}\in \{(p,0,\ldots,0,q)|p,q \in \{0,1\}, \; p+q=1 \};\\
\kappa-\nu, & \mathbf{j}\in \{(p,p,0,\ldots,0,q,q)|p,q \in \{0,1\}, \; p+q=1 \}; \\
\xi_2, & \mathbf{j}\in \{(0,p,0,\ldots,0,q,0)|p,q \in \{0,1\}, \; p+q=1 \};   \\
\mu-\nu,  & \mathbf{j}\in \{(\underbrace{0,\ldots,0}_{m},1,0,1,\underbrace{0,\ldots,0}_{n})|m,n \geq 0, \; m+n=d-3\}; \\
\nu,  & \mathbf{j}\in \{(\underbrace{0,\ldots,0}_{m},1,1,1,\underbrace{0,\ldots,0}_{n})|m,n \geq 0, \; m+n=d-3\};   \\
\kappa-2\nu, & \mathbf{j}\in \{(\underbrace{0,\ldots,0}_{m+1},1,1,\underbrace{0,\ldots,0}_{n+1})|m,n \geq 0, \; m+n=d-4\};  \\
\xi_3, & \mathbf{j}\in \{(\underbrace{0,\ldots,0}_{m+2},1,\underbrace{0,\ldots,0}_{n+2})|m,n \geq 0, \; m+n=d-5\},
\end{cases}
\end{equation}
where 
\begin{equation}
\xi_i=
\begin{cases}
 1/d-\kappa-\mu+\nu, & i=1; \\
1/d-2\kappa-\mu+2\nu,& i=2; \\ 
1/d-2\kappa-2\mu+3\nu, & i=3.
\end{cases}
\end{equation}

We claim that it suffices to check that the sets in the partition that are contained in $\cup_{i=1}^d A_i$ have non-negative probabilities (These form all binary vectors listed in \eqref{binvec1} and all except the one with $(p,q)=(0,0)$ in \eqref{binvec2}). This is so since $\Pr{A_i}=1/d$, for $i=1,\ldots,d$, implies by Boole's inequality that we have $\Pr{\cap_{i=1}^d A_i^\mathsf{c}}\geq 0$, and hence this latter probability can be assigned a value to make the probabilities of the sets in the partition add up to one. Towards this end, we note that \eqref{defnkappa} and \eqref{cond2dep} trivially imply 
\[
0\leq \kappa-2\nu \leq \kappa-\nu; \; 0\leq \nu,\mu-\nu; \text{ and }\xi_3\leq \xi_2\leq \xi_1.
\]
So all that remains to be shown is that $\xi_3\geq 0$. Towards this end, note that
\[
\xi_3=\begin{cases}
\frac{1}{d}-2\kappa+\mu=\frac{1}{d}(1-2\alpha+\beta), &\kappa\geq 2\mu;\\
\frac{1}{d}-\frac{\kappa}{2} - 2\mu=\frac{1}{2d}(2-\alpha-4\beta), &\kappa< 2\mu.
\end{cases}
\]
This combined with \eqref{cond2dep} yields $\xi_3\geq0$. The proof of the sufficiency part is now complete. 
\end{proof}

\begin{remark}\label{ma2-d<6}
For $d=3,4,$ and $5$, the necessary and sufficient conditions on $\alpha, \beta$ for a matrix of the form in \eqref{2dep} to be in $\mathcal{T}_d$ are listed in the table below. These were derived by using the probabilistic method discussed above.
\begin{center}
    \begin{tabular}{| l | l | }
    \hline
	$d$ & Conditions \\ \hline \hline
    3& $\alpha \geq 0$; \;$0 \leq \beta \leq 1 $; \;$2 \alpha-\beta \leq 1$.  \\ \hline
    4& $\alpha \geq 0$; \;$\beta \geq 0$; \;$\alpha+\beta \leq 1$; \;$2 \alpha-\beta \leq 1$. \\ \hline
    5& $\alpha \geq 0$; \;$0\leq \beta \leq 1/2$; \;$\alpha+\beta \leq 1$; \;$2 \alpha-\beta \leq 1$.  \\  \hline
    \end{tabular}
\end{center}
\end{remark}

\begin{remark}
Here we present a class of two-dependent stochastic processes exhibiting upper tail dependence, which is a particular case of the general time series model presented in \citet{Zhang2006}. Towards constructing this process, let $(X_{i}, i \in \mathbb{Z})$ be an array of iid random variables with unit Fr\'{e}chet distribution. The two-dependent stochastic process $(Y_{i}, i \in \mathbb{Z})$ is then defined as 
\[
Y_i=  \max \{cX_{i-2}, bX_{i-1}, aX_{i}\}, \qquad i \in \mathbb{Z},
\]
where $a$, $b$, and $c$ are nonnegative constants. It is easily checked that $Y_i$, $i \in \mathbb{Z}$, also has unit Fr\'{e}chet marginals with the additional constraint $a+b+c=1$; we will impose this constraint in the following. Moreover, this process is also a simple max-stable process, see \citet{Fiebig2017}. 

The lag $1$ upper tail dependence coefficient, that is the tail dependence coefficient between $Y_i$ and $Y_{i+1}$, for $i\in \mathbb{Z}$, is given by 
\begin{equation*} 
\begin{split}
\bar{\chi}_{i,i+1}& = \lim_{u \uparrow \infty} \frac{\Pr{Y_i \geq u,Y_{i+1} \geq u}}{\Pr{Y_{i+1} \geq u}} \\
&= \lim_{u \uparrow \infty}\frac{1-e^{-\frac{a\wedge b}{u}}+1-e^{-\frac{b\wedge c}{u}}+o \left(1-e^{-\frac{1}{u}}\right)}{1-e^{-\frac{1}{u}}}=a\wedge b+b\wedge c.
\end{split}
\end{equation*}
Similarly, the lag $2$ tail dependence coefficient is given by 
\[
\bar{\chi}_{i,i+2} = \lim_{u \uparrow \infty} \frac{\Pr{Y_i \geq u,Y_{i+2} \geq u}}{\Pr{Y_{i+2} \geq u}}= \lim_{u \uparrow \infty}\frac{1-e^{-\frac{a\wedge c}{u}}+o \left(1-e^{-\frac{1}{u}}\right)}{1-e^{-\frac{1}{u}}}=a\wedge c.
\]
Since the process is a $2$-dependence model, the lag $k$ tail dependence coefficients for $k\geq 3$ equal $0$. 
It is easily checked that the set of lagged tail dependence coefficients attained by this model remains the same even without the constraint $a+b+c=1$, although the marginals would deviate from unit Fr\'{e}chet. 

In the notation of Proposition \ref{MA2}, we thus have $\alpha=a\wedge b+b\wedge c$ and $\beta= a\wedge c$.
As shown in Figure \ref{alphabeta_maxstable}, we can attain any valid $(\alpha,\beta)$ vector as specified in Proposition \ref{MA2} by suitably choosing the vector of parameters $(a,b,c)$. Moreover, this model gives a stochastic process indexed by $\mathbb{Z}$, whence confirming that the conclusion of Proposition \ref{MA2} holds for the infinite dimensional case ($d=\infty$) as well. In particular, in this remark, we have given an alternate proof for the sufficiency part of Proposition \ref{MA2}, which moreover extends its necessary and sufficient condition to the $d=\infty$ case. 
\end{remark}
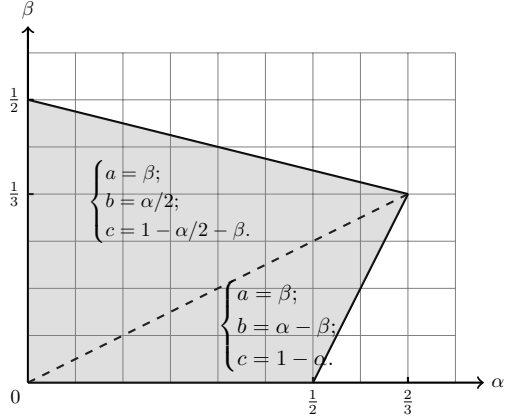
\begin{figure}
\centering
\begin{tikzpicture}[thick,scale=7.5, every node/.style={scale=0.75}]\footnotesize
 \pgfmathsetmacro{\xone}{0}
 \pgfmathsetmacro{\xtwo}{.8}
 \pgfmathsetmacro{\yone}{0}
 \pgfmathsetmacro{\ytwo}{.63}

\begin{scope}<+->;
  \draw[step=.08333cm,gray,very thin] (\xone,\yone) grid (.75,.583);
  \foreach \x/\xtext in { .5/{ \frac{1}{2}}, .66667/{\frac{2}{3}}}
  \draw[black,xshift=\x cm] (0,.01) -- (0,0) node[below] {$\xtext$};
  \foreach \y/\ytext in { .33333/\frac{1}{3},.5/{\frac{1}{2}}}
    \draw[black, yshift=\y cm] (.01,0) -- (0,0)   node[left] {$\ytext$};
 \draw[black] (0,0) node[anchor=north east] {$0$};
  \draw[thick,->] (\xone, 0) -- (\xtwo, 0) node[right] {$\alpha$};
  \draw[thick,->] (0, \yone) -- (0, \ytwo) node[above] {$\beta$};
   \draw[thick] (0, .5) -- (.66667, .33333) ;
     \draw[thick] (.66667, .33333) -- (.5, 0) ;  
      \draw[thick,dashed] (0, 0) -- (.66667, .33333); 
     \fill[gray!100,nearly transparent] (0,.5) -- (.66667,.33333) -- (.5,0) -- (0,0)--cycle;
     \draw[black] (.08,.32) node[right] {{ $\begin{cases}
     	a=\beta; \\
     	b=\alpha/2; \\
     	c=1-\alpha/2-\beta.
     \end{cases}$}};
     \draw[black] (.32,.1) node[right] {{\small $\begin{cases}
     	a=\beta; \\
     	b=\alpha-\beta; \\
     	c=1-\alpha.
     \end{cases}$}};
\end{scope}
\end{tikzpicture}
\caption{Coverage of $(\alpha,\beta)$ by Max-Stable Processes}
\label{alphabeta_maxstable}
\end{figure}

Consider a setup where we have two classes of risks; for example, each class may comprise of stocks of companies domiciled in the same country. While it is common to have significant information on intra-class joint behavior of risks, the same is not true for inter-class dependence. In such situations, it is natural to default to a non-informative assumption on inter-class dependence. In the case of tail dependence, and in particular in the context of TDMs, this translates to assuming a constant inter-class tail dependence coefficient. So this leads to the following problem: Given two $d_1$ and $d_2$ dimensional TDMs, say $\mathbf{T}_1$ and $\mathbf{T}_2$ respectively, for which values of $\gamma$ is the matrix 
\begin{eqnarray*}
\begin{bmatrix}
\mathbf{T}_1 & \gamma \mathbf{J}\\
\gamma \mathbf{J} & \mathbf{T}_2
\end{bmatrix}
\end{eqnarray*}
a valid $d_1+d_2$ dimensional TDM?  In the following example, we impose a equi-correlation structure on $\mathbf{T}_1$ and $\mathbf{T}_2$, and solve the combined problem instead of assuming that $\mathbf{T}_1$ and $\mathbf{T_2}$ are TDMs. This allows for exploring how between-class tail dependence is impacted by with-in class tail dependence, and vice versa.

\begin{example} \label{ex:two Sector}
Let $\mathbf{X}=(X_1,...,X_{d_1})$ and $\mathbf{Y}=(Y_1,...,Y_{d_2})$, for $d_1,d_2\geq 2$, be two random vectors such that each vector has a TDM with constant off-diagonal elements, which we denote by $\alpha$ and $\beta$, respectively.  As motivated above, we assume that the tail dependence coefficient between $X_i$ and $Y_j$ equals $\gamma$, for $1\leq i\leq d_1$ and $1\leq j\leq d_2$. These assumptions lead to a combined TDM for the $d_1+d_2$ dimensional random vector $(\mathbf{X},\mathbf{Y})$, denoted by $\mathbf{T}_d$, of the form
\begin{eqnarray}\label{2secmat}
                \begin{bmatrix}
                        1 &\alpha& \cdots & \alpha & \gamma &\gamma & \cdots & \gamma \\
                        \alpha &1& \cdots & \alpha & \gamma &\gamma & \cdots & \gamma \\
                        \vdots & \vdots &\ddots &\vdots &\vdots & \vdots &\ddots & \vdots \\
                        \alpha &\alpha & \cdots & 1 & \gamma & \gamma & \cdots & \gamma \\
                        \gamma & \gamma &  \cdots & \gamma & 1 & \beta & \cdots &\beta \\
                        \gamma & \gamma &  \cdots & \gamma & \beta & 1 & \cdots &\beta \\
                        \vdots & \vdots &\ddots &\vdots &\vdots & \vdots &\ddots & \vdots \\
                        \gamma & \gamma &  \cdots & \gamma & \beta & \beta & \cdots & 1
                \end{bmatrix}=
        \begin{bmatrix}
        (1-\alpha)\mathbf{I}_{d_1}+\alpha\mathbf{J}_{d_1} & \gamma \mathbf{J}_{d_1 \times d_2} \\
        \gamma \mathbf{J}_{d_2 \times d_1} & (1-\beta)\mathbf{I}_{d_2}+\beta\mathbf{J}_{d_2} \\
        \end{bmatrix}.
        \end{eqnarray}
The question of interest is the polytope of points $(\alpha,\beta,\gamma)$, say $\mathcal{P}_{d_1,d_2}$, that would make the above matrix a TDM. 

We begin by observing that Proposition \ref{Toep} implies that the matrix $(1-\xi)\mathbf{I}_{d}+\xi\mathbf{J}_{d}$ is in $\mathcal{T}_d$ for all $0\leq \xi \leq 1$. In particular, this implies the existence of a distribution $F_{d,\xi}$ which corresponds to the TDM $(1-\xi)\mathbf{I}_{d}+\xi\mathbf{J}_{d}$. Let us, for $0\leq \alpha,\beta\leq 1$, define the joint distribution $G_{\alpha,\beta}$ to be the product of $F_{d_1,\alpha}$ and $F_{d_2,\beta}$. By using the independence of the initial $d_1$ coordinates with the latter $d_2$ coordinates under $G_{\alpha,\beta}$, it is clear that its $d_1+d_2$ dimensional TDM is given by 
    \[
    \begin{bmatrix}
        (1-\alpha)\mathbf{I}_{d_1}+\alpha\mathbf{J}_{d_1} &  \mathbf{0}_{d_1 \times d_2} \\
         \mathbf{0}_{d_2 \times d_1} & (1-\beta)\mathbf{I}_{d_2}+\beta\mathbf{J}_{d_2} \\
        \end{bmatrix}.
    \]
This observation, the fact that the set of TDM is a convex set, and that the tail dependence coefficients are bounded below by $0$ permits us to equivalently pose the problem of interest as one of finding the upper bound for $\gamma$, say  $\gamma_u(\alpha,\beta)$, for values of $(\alpha,\beta)$ in $[0,1]^2$ for which the matrix in \eqref{2secmat} is a TDM. In other words, 
\[
\gamma_u(\alpha,\beta):=\sup\left\{\gamma \vert (\alpha,\beta,\gamma) \in \mathcal{P}_{d_1,d_2}\right\}.
\]

By Theorem \ref{BCTD}, the problem of determining $\gamma_u(\alpha,\beta)$ is equivalent to finding the upper bound of $d\gamma'$  such that the matrix $\mathbf{B}_d$ given by 
  \begin{eqnarray*}
        \begin{bmatrix}
        (1/d-\alpha')\mathbf{I}_{d_1}+\alpha'\mathbf{J}_{d_1} & \gamma' \mathbf{J}_{d_1 \times d_2} \\
        \gamma' \mathbf{J}_{d_2 \times d_1} & (1/d-\beta')\mathbf{I}_{d_2}+\beta'\mathbf{J}_{d_2} \\
        \end{bmatrix},
        \end{eqnarray*}
where $ \alpha'=\alpha/d,\beta'=\beta/d$, is in $\mathcal{B}_d$. Since $\mathcal{B}_d$ is closed, from Corollary \ref{BCprobcor}, we know that for every $\gamma'\leq \gamma_u(\alpha,\beta)/d$, there exists events $A_i,A_{d_1+j}$, $i=1,\ldots, d_1$, $j=1,\ldots, d_2$ on our canonical probability space $([0,1],\mathcal{L},\lambda)$ such that $\mathbf{B}_d$ has the form given in (\ref{BCprob}). Given the structure of $\mathbf{B}_d$, it is straightforward to see that $\mathcal{G}_{\mathbf{B}_d}$ is isomorphic to $\mathcal{S}_{d_1}\times \mathcal{S}_{d_2}$, with cardinality $d_1!d_2!$. 
The linear programming formulation that follows uses the reduction of Theorem \ref{reduce_thm}. Defining for $i=0,1,\ldots,d_1$ and $j=0,1,\ldots,d_2$,
\begin{multline*}
x_{i,j}:= \Pr{A_1^{\delta_1} \cap \cdots \cap A_{d_1}^{\delta_{d_1}}\cap A_{d_1+1}^{\delta_{d_1+1}} \cap \cdots \cap A_{d_1+d_2}^{\delta_{d_1+d_2}}}, \\
\hbox{ for } \mathbf{\delta}\in\{0,1\}^{d} \hbox{ satisfying } \sum_{j=1}^{d_1} \delta_{j}=i \hbox{ and } \sum_{k=1}^{d_2} \delta_{d_1+k}=j,
\end{multline*}

$\gamma_u(\alpha,\beta)/d$ is seen as the optimal value of the following LP problem:
\begin{equation}\label{two:sector:lp}
  \begin{array}{lll}
&\text{maximize} \quad\sum_{i=1}^{d_1}\sum_{j=1}^{d_2}\binom{d_1-1}{i-1}\binom{d_2-1}{j-1}x_{i,j}; &\\
&\text{subject to}  &\\   &\sum_{i=2}^{d_1}\sum_{j=0}^{d_2}\binom{d_1-2}{i-2}\binom{d_2}{j} x_{i,j}=\alpha'; & \sum_{i=0}^{d_1} \sum_{j=2}^{d_2} \binom{d_1}{i}\binom{d_2-2}{j-2} x_{i,j}=\beta'; \\
& & \\
&\sum_{i=1}^{d_1} \sum_{j=0}^{d_2} \binom{d_1-1}{i-1}\binom{d_2}{j} x_{i,j} =\frac{1}{d}; &\sum_{i=0}^{d_1} \sum_{j=1}^{d_2} \binom{d_1}{i}\binom{d_2-1}{j-1}x_{i,j} =\frac{1}{d}; \\
& & \\
                        &  \sum_{i=0}^{d_1} \sum_{j=0}^{d_2} \binom{d_1}{i}\binom{d_2}{j} x_{i,j} =1;  & x_{i,j} \geq 0, \quad 0\leq i\leq d_1, 0\leq j\leq d_2. 
\end{array}
\end{equation}

By using an LP solver (we used \texttt{IBM}\textsuperscript{\textregistered}  \texttt{ILOG CPLEX} in MATLAB\textsuperscript{\textregistered}), we computed $\gamma_u(\alpha,\beta)$ for values $(\alpha,\beta)$ over a uniform grid, say $\Delta$, in $[0,1]^2$. Since $\gamma_u(\alpha,\beta)$ is the upper bound for $\gamma$ such that $(\alpha,\beta,\gamma)$ belongs to the polytope $\mathcal{P}_{d_1,d_2}$, using a convex hull algorithm encapsulated in the function {\tt convhulln}\footnote{{\tt convhulln} is based on \citet{Barber1996}.} of MATLAB\textsuperscript{\textregistered}, on the set of points 
\[
\left\{\left(\alpha,\beta,\gamma_u(\alpha,\beta)\right) | (\alpha,\beta)\in\Delta \right\},
\]
we extract the extreme points of the top surface of $\mathcal{P}_{d_1,d_2}$. Also, the lower surface of $\mathcal{P}_{d_1,d_2}$ is the set $[0,1]^2\times\{0\}$, whose set of extreme points is  $\{0,1\}^2\times\{0\}$. Now using {\tt polymake} (see \citet{Gawrilow1997}) on the union of these sets of extreme points we extract the facets of the polytope $\mathcal{P}_{d_1,d_2}$. It is worth mention that to compute $\gamma_u(\cdot,\cdot)$ on a grid, solving the dual instead allows for the use of warm start, a technique which has shown in practice to provide significant speedup (for example, see  \citet{Vanderbei}). 

Using the above methodology we computed the polytope $\mathcal{P}_{d_1,d_2}$ for $(d_1,d_2)$ taking values in the set $\{(2,2),(2,4),(3,3),(4,4)\}$. The choice of these values for $(d_1,d_2)$ includes both symmetric and asymmetric cases, and varying values of $d_1+d_2$. In Figures \ref{matlab_2_2}, \ref{matlab_3_3}, \ref{matlab_2_4} and \ref{matlab_4_4} of the Appendix, we have images of these polytopes, and listed their facets and their vertices.  
\end{example}

\begin{remark}
We note that while the realization problem in general is likely not polynomial in complexity, restricted to the above class of parametric matrices the realization problem is polynomial in complexity. Towards showing this latter claim we note that both the number of variables and constraints in the LP stated in \eqref{two:sector:lp} are of the order $d^2$, and a simple change of variables results in coefficients bounded by $d^8$. The polynomial nature now follows by the algorithm for solving LPs given in \citet{Vaidya1989}, which solves the LP of \eqref{two:sector:lp} in $O(d^7\log(d))$ iterations.
\end{remark}

\begin{remark}\label{FreDeduc}
In the Appendix of \citet{Fiebig2017}, and in Tables 3.1 \& 3.3 of \citet{Strokorb2013}, a full description of the polytope of TDMs is specified for dimensions up to six. Since the above class has a parametrization that is linear in its parameters, one can deduce a $\mathcal{H}$-representation (see \citet{Fiebig2017}) of the polytope $\mathcal{P}_{d_1, d_2}$ from the facets of the polytope of the TDM in these dimensions. This $\mathcal{H}$-representation can be reduced to a facet representation using \texttt{Polymake}. This served as a check for the facet representations derived using the method of Example \ref{ex:two Sector}. However, when $(d_1,d_2)=(4,4)$, the lack of the facet representation of the polytope of TDMs in dimensions seven and up does not allow us to use this method of deduction, and hence the facet representation in Example \ref{ex:two Sector} for this case is particularly noteworthy. 
\end{remark}

\begin{remark}
When $d_1+d_2=3$ ({\it i.e.} $(d_1,d_2)\in\{(2,1),(1,2)\}$), the matrix \eqref{2secmat} is one of the following two forms:
\[
\begin{bmatrix}
1 & \alpha & \gamma \\
\alpha & 1 &\gamma \\
\gamma & \gamma & 1
\end{bmatrix}; \qquad
\begin{bmatrix}
1 & \gamma & \gamma \\
\gamma & 1 &\beta \\
\gamma & \beta & 1
\end{bmatrix}.
\]
By symmetry, the polytope of values of $(\alpha,\gamma)$ and $(\beta,\gamma)$ that generate TDMs are identical, say $\mathcal{P}_{1,2}$. The facets and vertices of $\mathcal{P}_{1,2}$ can be derived using the method described in Remark \ref{FreDeduc}, and are listed in Table \ref{2sector_rem}.
\end{remark}

\begin{table}[H] 
\caption{Description of the Polytope for $(d_1,d_2)=(1,2)$}
\label{2sector_rem}
\centering
\begin{tabular}{|c|c|} 
\hline
 Facets & Vertices\\
\hline
&\\
 $\beta, \gamma \geq 0$; $\beta \leq 1$; $\beta-2\gamma+1 \geq 0$ & $\{0\} \times \{0, 1/2\}$; $\{1\} \times \{0,1\}$ \\
&\\
\hline
\end{tabular}

\end{table}

\section{On Algorithms for Testing Membership of Arbitrary Matrices in $\calB_d$ and $\calT_d$}

In this section, we use some of the results derived above to lend insight into the practical side of membership testing. Currently, the most competitive algorithm for membership testing in $\calB_d$ and $\calT_d$ is the algorithm of \citet{Krause2018}, which is based on an LP formulation. In their algorithm, they choose the next incoming vertex for the reduced primal by solving a NP-hard Binary Quadratic Program (BQP) in every iteration. We show below that employing instead a particular continuous relaxation of this BQP, in essence replacing it by a polynomial time solvable problem, results in much improved running times. As expected from the above theoretical development of some parametric classes, when the decision problem restricted to a parametric class is polynomial time solvable, the latter approach is clearly superior, in practice as well, to a black-box algorithm. On the other hand, we demonstrate below using a certain experimental setup that even some randomly selected matrices in $\calB_d$ tend to have some inherent structure (see Table \ref{table:30p}), and identifying these can result in much shorter computational time - arguing against a black-box approach to membership testing. Finally, by employing a parametric class discussed above, we are able to answer even for large $d$ the following question:  do matrices close to the boundary of $\calB_d$ always make for difficult problem instances?  

We note that all of the computation of this section except Table \ref{table_cts} was done on a machine with a single Intel\textsuperscript{\textregistered} Quad Core\texttrademark{}  i7 2.7GHz processor with 16GB of RAM; for Table \ref{table_cts} we used one with a single Intel\textsuperscript{\textregistered} Quad Core\texttrademark{} i7 3.4GHz processor with 32GB of RAM. These computations used \texttt{IBM}\textsuperscript{\textregistered}  \texttt{ILOG CPLEX} LP solver on MATLAB\textsuperscript{\textregistered} Version R2018 b. Importantly, performance of the algorithm in \citet{Krause2018}, which we denote by KSSW, is reported by using their implementation which they very generously made available to us. Finally, we define two classes of instances from the test battery of \citet{Krause2018} that we use below for reasons of both objectivity and comparability. Their Problem Class 3 refers to matrices simulated from $\calB_d$ by generating a number, say $N$, uniformly between $d^2$ and $d^4$, randomly choosing $N$ vertices of $\calB_d$ from among the $2^d$, and then finding a random convex combination of these chosen vertices by choosing the weights using an exchangeable Dirichlet distribution with uniform marginals. Their Problem Class 5 refers to simulated matrices of the form $\mathbf{A}+\mathbf{B}/d$, where $\mathbf{A}$ is a matrix generated using Problem Class $3$, and $\mathbf{B}$ is the first vertex chosen to construct $A$ (or equivalently one chosen uniformly from among these vertices).

In Remark 3 of \citet{Krause2018}, it is observed that there are alternate LP formulations of the realization problem based on the choice of the objective function. The algorithm KSSW is based on a formulation that computes the minimum distance of a given matrix from $\calB_d$ in terms of the matrix max-norm, with zero distance being equivalent to membership in $\calB_d$. An alternate formulation is that of our Theorem \ref{FABC}. Importantly, they avoid/delay confronting the exponentiality of the number of variables in their algorithm by employing the column generation method, see pg. 249 of \citet{Krause2018}. This method starts with testing membership in the convex hull of a carefully selected set of $O(d^2)$ vertices, and sequentially adds vertices based on the {\em most violated constraint} of the dual. The use of the column generation method more than doubles the limits on $d$ from about $20$ to above $40$. 

A natural question then is whether alternate LP formulations can equally benefit from such optimization techniques. So we compared the formulation of Theorem \ref{FABC} along with suitably modified enhancements employed in \citet{Krause2018}, with the KSSW algorithm. For matrices in $\calB_d$, {\it i.e.} the positive instances, we found that the two formulations had comparable performance across dimensions, with no formulation having uniformly superior running times. In the case of test matrices outside $\calB_d$, since the formulation of Theorem \ref{FABC} has a trivial objective function, it is unable to use dual bounds (see Proposition 8 of \citet{Krause2018} and bound (5) of \citet{Lubbecke2011}\footnote{We note that this latter bound was employed by \citet{Krause2018} in their implementation of KSSW even though it is not mentioned in the published article.}) towards an early termination with a negative decision. This gives a decisive advantage to formulations incorporating non-trivial objective functions such as KSSW. To clearly bring out this advantage, we do rejection sampling to select $100$ negative instances from Problem Class $5$ on which the KSSW algorithm exits because of the use of dual bounds. Since the dual approximation bound of Proposition $5$ of \citet{Krause2018} is computed by KSSW before executing the column generation method, and since this is very successful on negative instances from Problem Class 5, for our current purpose we {\it switched off} this pre-check.  We chose lower dimensions because for higher $d$ such instances tend to be overwhelmingly in $\calB_d$. In Table \ref{dual-bound}, we report the average running times of the KSSW implementation and an implementation of the LP formulation underlying Theorem \ref{FABC} which employs the column generation method. The significantly lower average number of iterations for KSSW demonstrates the benefits of using dual bounds.
\begin{table}[ht]    
\caption{Comparison of LP With \& Without a Non-Trivial Objective: Negative Cases}      
\label{dual-bound}
\centering                             
\begin{tabular}{|c|c|c|c|c|}    
\hline                                     
\multicolumn{1}{|c|}{\multirow{2}{*}{$d$}} & \multicolumn{2}{c|}{Avg. Running Time (in s)} &\multicolumn{2}{c|}{Avg. No. of Iterations}\\ \cline{2-5}
\multicolumn{1}{ |c|  }{} & Theorem \ref{FABC}  & KSSW & Theorem \ref{FABC}  & KSSW \\
\hline
16  &17.6 &1.18 &119 & 1.53      \\
18 &84.8&6.50&199&2.39       \\
20  & 233 &26.1 & 200 &7.22       \\
\hline                      
\end{tabular}                              

\end{table}


One important part of the column generation method is the introduction of new variables at each iteration, and this is done by finding the most violated dual constraint. In the case of the algorithm of \citet{Krause2018}, they notice that this sub-problem can be formulated as a BQP of the form,
\begin{equation}\label{eq:BQP}
\max_{\mathbf{p} \in \{0,1\}^d}  \mathbf{p}^\mathsf{T} \mathbf{G}  \mathbf{p},
\end{equation}
where $\mathbf{G}$ is some arbitrary symmetric matrix. Even though the BQP is a NP-hard problem (see \citet{padberg1989}), 
\citet{Krause2018}  observe that on CPLEX the above formulation significantly reduces the compute time compared to a full evaluation of all of the $2^d$ constraints. 
Driven by the fact that it is a heuristic prescription to choose the new vertex by using the most violated constraint, and that the BQP is NP-hard, we explored a continuous relaxation of the BQP. But a straightforward relaxation of the form   
\[
\max_{\mathbf{p} \in [0,1]^d}  \mathbf{p}^\mathsf{T} \mathbf{G}  \mathbf{p},
\]
which is a Quadratic Programming (QP) problem, moves us from one NP-hard problem to another if the $\mathbf{G}$ is not negative semidefinite, see Theorem 2.5.4 of \citet{Sahni1974}, and also \citet{Pardalos1991}. Of course, if $\mathbf{G}$ is negative semidefinite then the original BQP in \eqref{eq:BQP} is trivial. Towards this end we note that the original BQP is easily seen to be equivalent to
\[
 \max_{\mathbf{p} \in \{0,1\}^d}   \mathbf{p}^\mathsf{T} (\mathbf{G}- \text{diag}(\mathbf{f}))  \mathbf{p} + \mathbf{f}^\mathsf{T}   \mathbf{p}.
\]
Now choosing $\mathbf{f} = (\max_{i} \lambda_i+\varepsilon) \mathbf{1}_{d \times 1}$, where $\lambda_i$, for $i=1,\ldots,d$, are the eigenvalues of $\mathbf{G}$ and $\varepsilon$ is a positive real (chosen to be small, for example $10^{-8}$), makes $\mathbf{G}- \text{diag}(\mathbf{f})$ a negative definite matrix\footnote{In the implementation of KSSW, \citet{Krause2018} use a similar positive definite version of their BQP. }. This leads to our continuous relaxation 
\[
 \max_{\mathbf{p} \in [0,1]^d}   \mathbf{p}^\mathsf{T} (\mathbf{G}- \text{diag}(\mathbf{f}))  \mathbf{p} + \mathbf{f}^\mathsf{T}   \mathbf{p},
\]
which is a QP problem with a negative semidefinite matrix, and such QPs can be solved in polynomial time, see \citet{Kozlov1979}, \citet{kapoor1986fast} and \citet{ye1989extension}. The choice of $\mathbf{f}$ is guided by Theorem 2 of \citet{Hammer1970}, as we round the solution to get an {\it approximate} solution to the original BQP of \eqref{eq:BQP}. Now by using this latter approximate solution to the BQP, instead of the actual solution, we fall short of identifying the most violated constraint of the BQP. This has a few implications. First, if such an identified condition is not a violation then one should solve the BQP; note that such conditions could be already part of the reduced dual. Second, the dual bounds such as Proposition 8 of \citet{Krause2018} and the bound in (5) of \citet{Lubbecke2011} could no longer be valid. Since, as demonstrated above in Table \ref{dual-bound}, such bounds have found to be quite useful in some negative instances, we recommend that every so many iterations one should solve the BQP. 

To investigate the empirical performance of the above alternate solution towards identifying the next incoming vertex, we use the Problem Class 3 of positive instances as described above. We randomly generated $100$ such instances for dimensions $25,30,35$, and $40$, but only $20$ for dimension $45$ as even for this small sample size the total compute time for dimension $45$ was above $18$ hours\footnote{For dimension $45$, we increase the default limit on the running time in KSSW from $30$ minutes to $90$ minutes.}. These instances served as inputs to the KSSW algorithm as well as its above suggested modification which we denote by KSSWcr. In Table \ref{table_cts} and Figure \ref{figure_cts} below, we report the average running time, average number of iterations, and the average time spent in the sub-problem of identifying the next incoming vertex. Note that while it is expected for the number of iterations used by KSSWcr to be higher than that of KSSW, as we are no longer guaranteed to identify the most violated constraint, this phenomenon is seen only for lower dimensions.  
This is so as quite unexpectedly we see that the growth in the number of iterations used by KSSWcr is quite muted compared to that by KSSW. Moreover, as expected we see a significant reduction in the time spent on the sub-problem which results in a close to four fold decrease in the running time at dimension $45$. In other words, in higher dimensions the solution of the BQP becomes the bottleneck (and not the solution of the LPs in the column generation process), rendering  beneficial the speeding up of the computation of the most violating constraint. 


\begin{table}[htbp]
\caption{Comparison of the use of BQP versus its Continuous Relaxation}
\label{table_cts}
\centering
\begin{tabular}{|c|c|c|c|c|c|c|}
\hline
\multicolumn{1}{|c|}{\multirow{2}{*}{$d$}} & \multicolumn{2}{c|}{Avg. Running Time (in s)} & \multicolumn{2}{c|}{Avg. No. of Iterations}& \multicolumn{2}{c|}{Time Spent in BQP/QP (in s)} \\ \cline{2-7}
\multicolumn{1}{ |c|  }{} & KSSW & KSSWcr & KSSW & KSSWcr& KSSW & KSSWcr\\ \hline
 25 &10.06&6.096 &55.41&71.49&4.693&0.438 \\
30&53.03&32.96&207.8&271.8&24.38&1.86\\
35&  199.3&103.6&  369.7&  431.4&  102.9&    5.08\\
40&  764.7&280.6&586.9& 584.0&497.9&6.21\\
45&  2613&669.9 &   851.7    &758.5&    1937&    13.2\\
\hline
\end{tabular}

\end{table} 

\begin{figure}[ht]
\centering
	\includegraphics[width=11.5 cm]{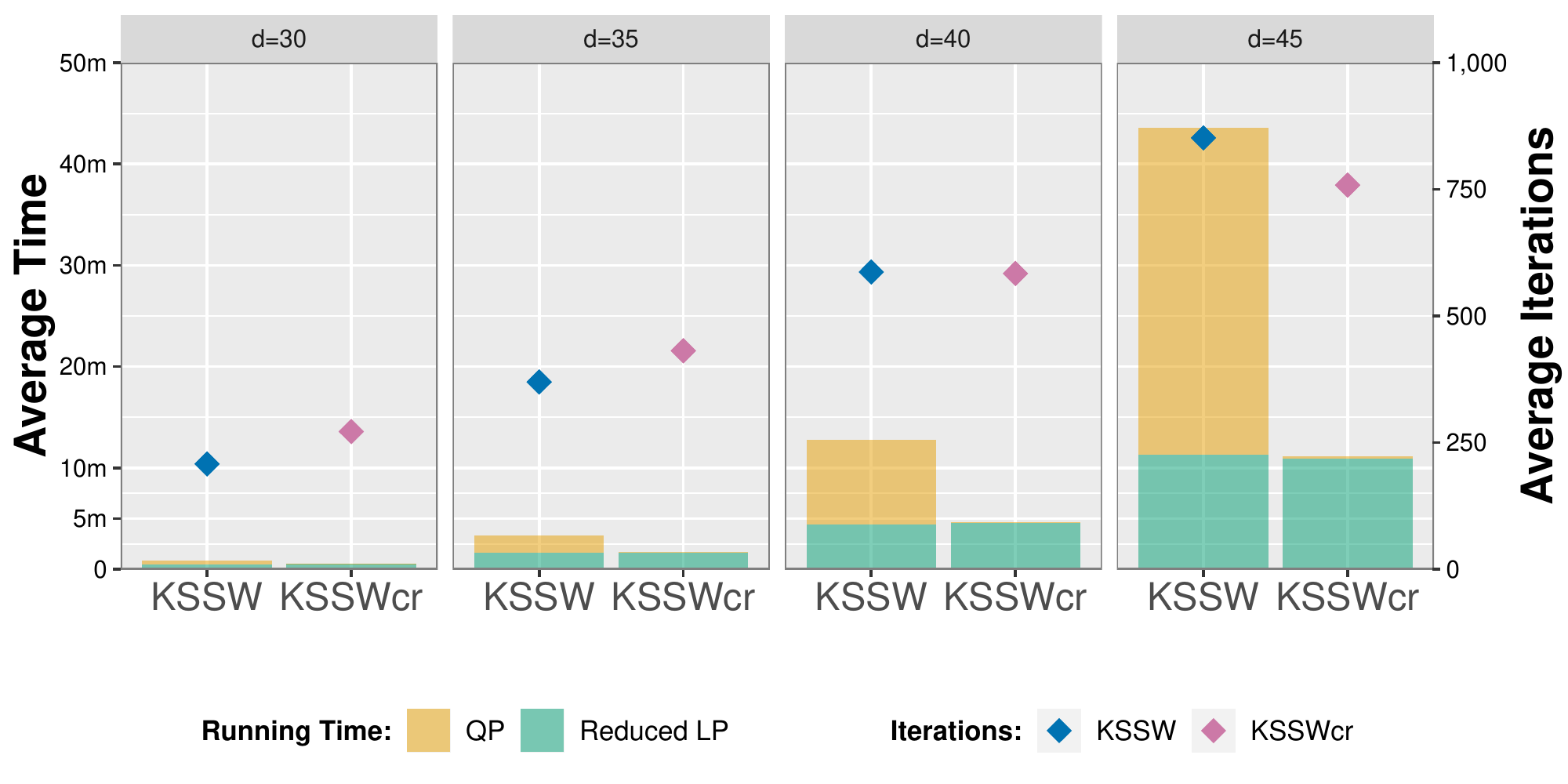}
	\caption{Graphical Representation of Table \ref{table_cts} - BQP versus its Continuous Relaxation} 
	\label{figure_cts} 
\end{figure} 

It is worth mention that there are two approaches to membership testing. One that is represented by using a general LP formulation, either that based on Theorem \ref{FABC} or the KSSW algorithm, or any other such version. The idea being to design a general algorithm that could cater to all inputs without further human intervention, in other words a black-box approach to membership testing. The other approach being to make use of structures inherent in the modeling, and to design a customized algorithm. The obvious advantage of a black-box approach is that it does not require much time of the modeler. On the other hand customized solutions can have huge performance advantage in the presence of useful inherent structures. Consider for example the highly symmetric $3$-parameter class presented in Example \ref{ex:two Sector}. In this case, the customized LP formulation of \eqref{two:sector:lp} is $O(d^2)$ in size versus the exponential description of the general LP formulation. Hence, in particular, the customized solution for the parametrized form of Example \ref{ex:two Sector} implies that the membership testing problem in this restricted setting is in P. In Table \ref{table_2sector}, we report average run times of $100$ instances per dimension per decision (positive or negative). These instances were generated by assigning $d_1=\lceil d/2 \rceil$ and $d_2=\lfloor d/2 \rfloor$, uniformly sampling $(\alpha,\beta,\gamma)$ from $[0,1]^3$ and selecting the first $100$ positive and negative instances for each dimension. We note that the KSSW algorithm is set to time out at the 30 minute mark per instance, as reported in \citet{Krause2018}. Hence, if any instance for a given dimension and decision timed out, then we reported both the percentage of cases it timed out on as well as the average running times among those that terminated. Note that while the KSSW algorithm times out for the negative (resp., positive) cases starting with dimension $20$ (resp., $30$), the customized algorithm of Example \ref{ex:two Sector} can effortlessly handle even dimension $1,000$.   

\begin{table}[htbp]
\caption{Running Time on Instances of Example \ref{ex:two Sector}}
\label{table_2sector}
\centering
\begin{threeparttable}
\begin{tabular}{|c|l|l|l|l|c|c|}
\hline
\multicolumn{1}{|c|}{\multirow{3}{*}{$d$}} &\multicolumn{4}{ c|  }{ Avg. Running Time (in s) }\\ \cline{2-5} 
\multicolumn{1}{ |c|  }{} & \multicolumn{2}{c|}{Algo. of Ex. \ref{ex:two Sector}} & \multicolumn{2}{c|}{KSSW (Proportion of Timeouts)} \\ \cline{2-5}
\multicolumn{1}{ |c|  }{} & Positive & Negative & Positive & Negative \\ \hline
5&  0.00054 &0.00053    &0.02477& 0.01695 \\
10 & 0.00057&0.00066   &0.03108&  0.01742 \\
15  & 0.00064 &   0.00068    &0.43523&  0.20645  \\
20  &   0.00071&   0.00077 & 0.74621& 19.538\tnote{$\blacktriangle$} \; (3\%)\\
25  &  0.00091&0.00094&3.4348&88.408\tnote{$\blacktriangle$} \; (33\%)\\
30  & 0.00091&0.00107&15.339\tnote{$\blacktriangle$} \; (1\%)&-\\
35  & 0.00127&0.00144&51.291\tnote{$\blacktriangle$} \; (1\%)&-\\
40  & 0.00156&0.00167& 94.193\tnote{$\blacktriangle$} \; (5\%)&-\\
50 &0.00228&0.00227&-&-\\
100 &0.00658&0.00662&-&-\\
1000 &3.94& 4.20&-&- \\
\hline 
\end{tabular}
\begin{tablenotes}
\item [$\blacktriangle$] The average running time excluding timeouts 
\end{tablenotes}
\end{threeparttable}

\end{table}

\begin{table}[ht]  
\caption{Running Time for $\calB_{30}$ Instances Near the {\em Boundary} - Example \ref{ex:two Sector}}        
\label{3para_positive}
\centering     
\begin{threeparttable}
\begin{tabular}{|r|c|c|}                      
\hline                                                                       
\multicolumn{1}{|c|}{\multirow{2}{*}{Distances to $\gamma_u$}}&\multicolumn{2}{c|}{ Avg. Running Time (in s) } \\ \cline{2-3}
\multicolumn{1}{ |c|  }{}&  Algo. of Ex. \ref{ex:equi-correlation} & KSSW (Proportion of Timeouts) \\ \cline{2-3} 
\hline
    50\% $\gamma_u$ &  0.0009 &  7.16  \\
    25\% $\gamma_u$ &  0.0009 &   20.6  \\
    12.5\% $\gamma_u$ &  0.0009  & 29.2 \\
    6.25\% $\gamma_u$ &  0.0009   & 68.7\tnote{$\blacktriangle$} \; (1\%) \\
    3.13\% $\gamma_u$ &  0.0009   & 120\tnote{$\blacktriangle$} \; (9\%)\\
\hline                
\end{tabular}       
\begin{tablenotes}
\item [$\blacktriangle$] The average running time excluding timeouts 
\end{tablenotes}
\end{threeparttable}

\end{table}

Another phenomenon we observed is that algorithms designed for membership testing of general matrices are non-robust in the sense that the running times tend to, but not always, increase as these matrices approach the boundary of $\calB_d$. We first demonstrate this phenomenon in the case of matrices of the form considered in Example \ref{ex:two Sector}. Since the negative cases start to time out at dimension $20$, we restrict ourselves to positive cases near the boundary with the dimension set at $30$, and as a proxy for distance to boundary we used $$\frac{ \gamma_u(\alpha,\beta)-\gamma}{\gamma_u(\alpha,\beta)}.$$ We generated our test cases by setting $d_1=d_2=d/2=15$ and uniformly generating $100$ values of $(\alpha,\beta)$ from $[0,1]^2$; for each $(\alpha,\beta)$ so generated, we used the LP formulation of \eqref{two:sector:lp} to create five test cases $(\alpha,\beta,(1-2^{-i})\gamma_u(\alpha,\beta))$, for $i=1,\ldots,5$.  For each $i=1,\ldots,5$, we report the average running time of the KSSW algorithm among the $100$ cases in Table \ref{3para_positive}. The KSSW algorithm starts to time out at the 30 min. mark on some test cases beginning with $i=4$, and so we reported only the average running time among those that did not time out to show the non-robustness in this metric as well. 

To demonstrate the above phenomenon graphically, we chose the instances from Example \ref{ex:equi-correlation}. We uniformly generate $1,000$ points from the polytope of the equi-correlation matrices of Example \ref{ex:equi-correlation} for $d=30$ using the function {\tt cprnd} (see \citet{cprnd}). In Figure \ref{1000}, we report the running times of KSSW by using a heat map, in which we observe that proximity to the boundary does not always lead to increased running times. In particular, we see that near the part of the boundary which represents comonotonic dependence, {\it i.e.} $\alpha=\beta$, running times are on the lower side in contrary to expectation. We note that the mean time of the $1,000$ instances equals $41.95$, with standard deviation of $46.19$, minimum of  $0.3608$ and maximum of $416.6$. 

\begin{figure}[ht]
\centering
	\includegraphics[width=12 cm]{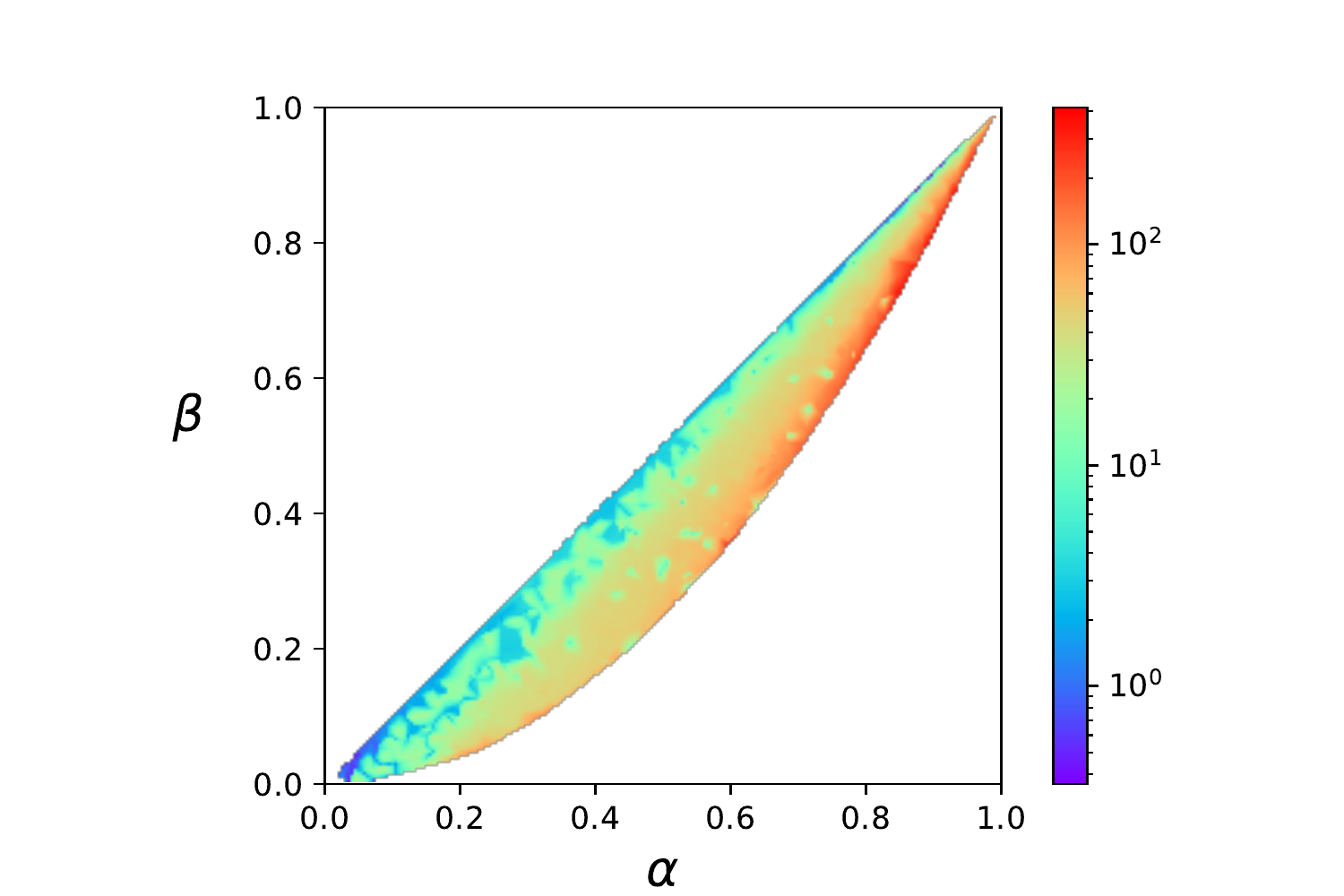}
	\caption{Running Time for $\calB_{30}$ Instances from Example \ref{ex:equi-correlation}} 
	\label{1000}
\end{figure}

Finally, we wish to point out that even arbitrary matrices may have structures that could be useful to reduce running times. To demonstrate this, we randomly generated five, dimension $30$ matrices, from Class $3$ of \citet{Krause2018} using their simulation methodology described above. For each such matrix, we randomly permuted the rows and columns $100$ times, and recorded the running times of KSSW on each of these permutations. In Table \ref{table:30p}, we show some summary statistics of these running times for each of the five instances. The minimum time reflects the success of the choice of the initial set of vertices as described in \citet{Krause2018} for some permutations but not for most. Seemingly, the problem of identifying such a permutation is of exponential complexity. But in practice, this is likely to arise because of choices made by the modeler; this lends further support to a more customized, modeler guided approach to the realization problem.

\begin{table}[ht]                                                   \caption{Running Times for Random Permutations of Instances with $d=30$}         
       
\label{table:30p}              
\centering                                                                    
\begin{tabular}{|c|c|c|c|c|c|}                                                  
\hline                                                                    
\multicolumn{1}{|c|}{\multirow{2}{*}{Cases}} &\multicolumn{4}{ c|  }{Running Time (in s)} & \multicolumn{1}{c|}{\multirow{2}{*}{\# Less than $10$s }} \\ \cline{2-5}
\multicolumn{1}{ |c|  }{} &  Mean & Min. & Max. & Standard Deviation &\multicolumn{1}{ c|  }{}\\
\hline
    1 &  62.6  & 7.19& 89.4& 12.9&4\\
    2 &  67.5  &  7.46& 87.2& 9.49&1\\
    3  &  67.9 &  7.47& 94.0 &16.2&5\\
    4  &  70.0  & 7.77& 94.4 &14.8&4\\
    5 &  71.2 & 7.75& 99.8& 10.2&1\\
\hline                                                                         
\end{tabular}                                                   
\end{table}

\begin{figure}[ht]
\centering
	\includegraphics[width=10 cm]{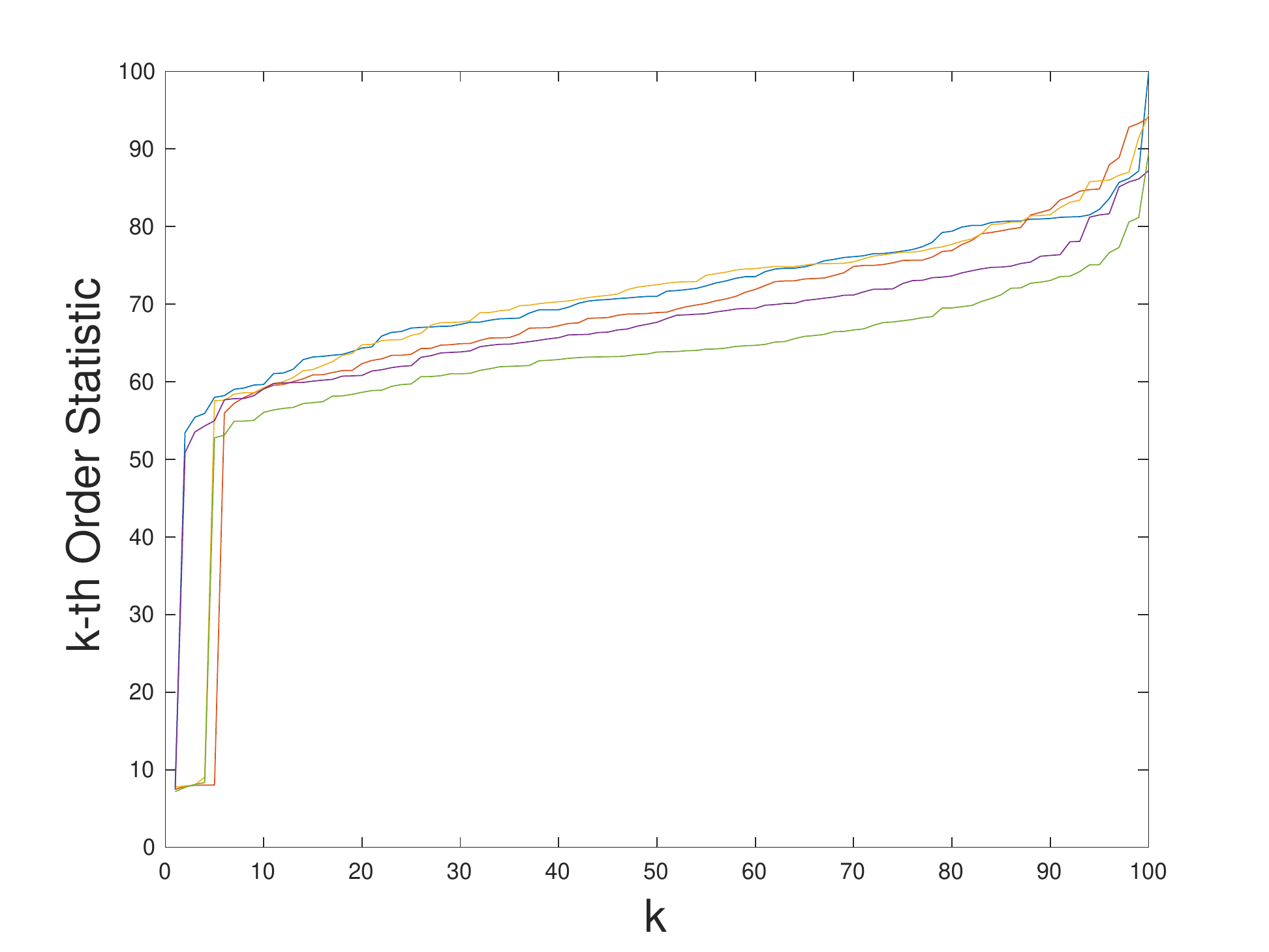}
	\caption{Running Times for Random Permutations of Instances with $d=30$} 
	\label{30p}
\end{figure}

\section{Conclusion and Discussion}
In this study, the focus was on developing a methodology for the TDM realization problem restricted to parametric classes. A refinement of a natural LP formulation, and its solution by the ellipsoid method establishes a connection with the NP-complete max-cut problem on a real weighted undirected graph. In a sense this can also be seen as an alternate method to establish that the cut polytope is related to the realization problem. As mentioned in section 4, this suggests that a fruitful direction by which to identify classes of TDMs that will admit a polynomial time algorithm for the realization problem is to identify those for which a polynomial time separation oracle can be constructed. Facing the unlikely prospect that a polynomial time algorithm can be found for the general realization problem, we turned our focus on techniques by which the realization problem can be solved for lower dimensional parametric classes of TDMs. We showed through some classes of TDMs that sparsity and symmetry can be exploited to reduce the complexity of the LP formulation to polynomial time.  

In the case of equi-correlation BCM matrices (Example \ref{ex:equi-correlation}), the limiting polytope has an area of $1/6$, and the area of the polytope for any finite $d$ is given by $1/6+(6d)^{-1}$. So the limiting polytope is in a sense a $O(d^{-1}$) approximation. In cases where the finite $d$ polytope are computationally hard to describe, it would be interesting to see if we can still find approximations which are provably good. In this connection, it is interesting to note that in the two-dependence class of Proposition \ref{MA2}, the polytopes coincide for $d\geq 6$. We conjecture that the latter stabilization of the constraint polytopes holds for all $m$-dependence TDMs of Proposition \ref{prop:m-dependence}. 

A noteworthy aspect of Proposition \ref{Toep}, which provides sufficient conditions for the constraint polytope for Toeplitz TDM matrices, is that while the conditions are not necessary they are satisfied in all of the TDM parametric classes we considered. This suggests that practically useful sufficient conditions for parametric classes can be constructed using the methods of this paper, even if necessary and sufficient conditions are out of reach. In view of the previous discussion, it will be interesting to derive bounds on the ratio of the volume of the polytope described by the sufficient conditions of Proposition \ref{Toep} to the volume of the exact polytope. Also, it would be interesting to explore if the Toeplitz class of TDMs would admit a polynomial time algorithm. Finally, the constraint polytopes of the parametric classes considered in this paper exhibit various properties with some of them emphasized here; we believe it would be both theoretically interesting and practically relevant to explore the generality in which these properties hold true.  


Part of this work was done independent of \citet{Fiebig2017}, and also \citet{Strokorb2013}. Moreover, after presenting the results of this paper in the {\it 21st International Congress on Insurance: Mathematics and Economics - IME 2017}, \citet{Krause2015} was brought to our attention; also, during the refereeing process we became aware of a more extensive published version \citet{Krause2018}. While this article, \citet{Fiebig2017} and \citet{Krause2018} deal with the realization problem, and there is some overlap, the main focus of these articles are distinct as explained below. 

In \citet{Fiebig2017} a study of the tail correlation function (TCF) realization problem is conducted, with TCF being the analogue of TDMs for stochastic processes. They show that a function is a valid TCF \iff its restriction to finite sets is a valid TCF, {\it i.e.} they reduce the TCF realization problem to that of the TDM.  Independent of \citet{Embrechts2016}, they establish the connection between $\calT_d$ and BCMs, and derive many results on the geometric structure of $\calT_d$. They observe a fundamental connection between $\calT_d$ and the correlation and cut polytopes studied in the area of computational geometry (see \citet{Deza1997}); this allows them to derive an explicit facet representation of $\calT_d$, for $d\leq 6$. \citet{Krause2018} is akin to us motivated by the search for an algorithm to determine membership in $\calT_d$, which they as well reduce to determining membership in $\calB_d$. They like \citet{Lee1993} consider an LP formulation of this problem by using the vertex representation of $\calB_d$, but with a non-trivial objective function. Since their LP formulation is exponential in description, and moreover since \citet{kaibel2015short} shows that so will any other LP formulation, their focus is mainly on techniques that speed up their LP formulation in practice.

\vspace{1 cm}
\textbf{Acknowledgements.}The authors thank Paul Embrechts for bringing our attention to \citet{Embrechts2016}, which led to our interest in the realization problem. We thank an associate editor for bringing our attention to \citet{Krause2017}, and an area editor and two referees for many thoughtful suggestions. We thank Jonas Schwinn and Ralf Werner for graciously sharing their  well executed implementation of the KSSW algorithm without which parts of this work would not have been possible. Also, we would like to thank Sam Burer and Ruodu Wang for fruitful discussions. The first author would like to acknowledge with gratitude the support from  a Society of Actuaries’ Center of Actuarial Excellence Research Grant.

\bibliographystyle{spmpsci}      


\appendix{}

\section{Appendix}

\begin{figure}[H]
  \centering
 \hspace{-1.6cm}
   \begin{minipage}[t]{0.3\textwidth}
   \vspace{-0.5 cm}
   {\large
  \[
  \mathbf{B}:=\begingroup
\renewcommand*{\arraystretch}{1.3}
  \begin{bmatrix}
  \frac{1}{2} & \frac{1}{6} &\frac{1}{6} &\frac{1}{6} &\frac{2}{9} &\frac{2}{9}  \\
  \frac{1}{6} & \frac{1}{2} &\frac{1}{6} &\frac{1}{6} &\frac{2}{9} &\frac{2}{9}  \\
  \frac{1}{6} & \frac{1}{6} &\frac{1}{3} &\frac{1}{6} &\frac{2}{9} &\frac{2}{9}  \\
  \frac{1}{6} & \frac{1}{6} &\frac{1}{6} &\frac{1}{3} &\frac{2}{9} &\frac{2}{9}  \\
  \frac{2}{9} & \frac{2}{9} &\frac{2}{9} &\frac{2}{9} &\frac{1}{3} &\frac{1}{8}  \\
  \frac{2}{9} & \frac{2}{9} &\frac{2}{9} &\frac{2}{9} &\frac{1}{8} &\frac{1}{3}  \\
  \end{bmatrix}
  \endgroup
  \]}
  \end{minipage}
  \qquad 
  \begin{minipage}[t]{0.6\textwidth}
  {\small
  \begin{verbatim}
G = Graph(sparse=True,loops=True)
# Adding edges corresponding to off-diagonal elements of B
G.add_edges([(1,2,'a'),(1,3,'a'),(1,4,'a'),
(2,3,'a'),(2,4,'a'),(3,4,'a'),(5,6,'b')] )
# Adding self-loops corresponding to the diagonal elements of B
G.add_edges([(1,1,'c'),(2,2,'c'),
(3,3,'d'),(4,4,'d'),(5,5,'d'),(6,6,'d')] )
# Computing the generators of the automorphism group
H = G.automorphism_group(edge_labels=True) 
# Computing N(B) using Polya's Enumeration Theorem
H.cycle_index().expand(2)(1,1)
  \end{verbatim}}
  \end{minipage}
 \caption{A Matrix and its Associated SageMath Code to Compute $N(\mathbf{B})$}
 \label{fig:sagemath}
\end{figure}

\begin{figure}[H]
  \centering
  \begin{minipage}[b]{0.4\textwidth}
    \begin{tikzpicture}[tdplot_main_coords, scale=2, >=stealth]   
    \draw[fill=gray,opacity=0.5] (0,0,0) -- (1,0,0) -- (1,0,.5) -- (0,0,.5) -- cycle;
    \draw[fill=gray,opacity=0.5] (0,0,0) -- (0,0,.5) -- (0,1,.5) -- (0,1,0) -- cycle;
    \draw[fill=gray,opacity=0.5] (0,0,.5) -- (1,0,.5) -- (1,1,1) -- cycle;
    \draw[fill=gray,opacity=0.5] (0,0,.5) -- (1,1,1) -- (0,1,.5) -- cycle;   
    \draw[dashed] (1,0,0) -- (1,1,0);
    \draw[dashed] (0,1,0) -- (1,1,0);
    \draw[dashed] (1,1,1) -- (1,1,0);
    \draw[black] (0,0,0) node[below] {$(0,0,0)$};
    \draw[black] (1,1,1) node[above] {$(1,1,1)$};
    \draw[black] (1,0,0) node[left] {$(0,1,0)$};
    \draw[black] (0,1,0) node[right] {$(1,0,0)$};
  \end{tikzpicture} 
    \caption{The Constraint Polytope when $(d_1,d_2)=(2,2)$}
    \label{matlab_2_2}
  \end{minipage}
  \quad
  \begin{minipage}[b]{0.4\textwidth}
  {\bf Polytope in Parameter Space:}
  \begin{description}
  \item [Facets:]
  \begin{align*}
    \alpha, \beta,\gamma &\geq 0;\quad &\alpha,\beta \leq 1; \\
  \alpha-2\gamma+1 &\geq 0;\quad  &\beta-2\gamma+1 \geq 0.
  \end{align*}
  \item [Vertices: $(\alpha,\beta, \gamma)$]
    \begin{align*}
    \{0,1\}^2\times\{0\};& & (1,1,1);& &(0,0,1/2);& \\
    (0,1,1/2);& &(1,0,1/2).& & & & & \\
  \end{align*}
  \end{description}
  \vfill
  \end{minipage}
\end{figure}

    \begin{figure}[H]
  \centering
  \begin{minipage}[b]{0.4\textwidth}
  \begin{tikzpicture}[tdplot_main_coords, scale=2, >=stealth]   
    \draw[fill=gray,opacity=0.5] (0,0,0) -- (1,0,0) -- (1,0,1/3) -- (0,0,1/3) -- cycle;
    \draw[fill=gray,opacity=0.5] (0,0,0) -- (0,0,1/3) -- (0,1,1/3) -- (0,1,0) -- cycle;
    \draw[fill=gray,opacity=0.5] (0,0,1/3) -- (1,0,1/3) -- (1,1/3,2/3) -- cycle;
    \draw[fill=gray,opacity=0.5] (0,0,1/3) -- (1/3,1,2/3) -- (0,1,1/3) -- cycle;
    \draw[fill=gray,opacity=0.5] (0,0,1/3) -- (1,1,1) -- (1,1/3,2/3) -- cycle; 
     \draw[fill=gray,opacity=0.5] (0,0,1/3) -- (1/3,1,2/3) -- (1,1,1) -- cycle;   
    \draw[dashed] (1,0,0) -- (1,1,0);
    \draw[dashed] (0,1,0) -- (1,1,0);
    \draw[dashed] (1,1,1) -- (1,1,0);
        \draw[black] (0,0,0) node[below] {$(0,0,0)$};
    \draw[black] (1,1,1) node[above] {$(1,1,1)$};
    \draw[black] (1,0,0) node[left] {$(0,1,0)$};
    \draw[black] (0,1,0) node[right] {$(1,0,0)$};
  \end{tikzpicture} 
    \caption{The Constraint Polytope when $(d_1,d_2)=(3,3)$}
    \label{matlab_3_3}
  \end{minipage}
  \quad
  \begin{minipage}[b]{0.4\textwidth}
  {\bf Polytope in Parameter Space:}
  \begin{description}
  \item [Facets:]
  \begin{align*}
    \alpha, \beta,\gamma \geq 0;& &\alpha,\beta \leq 1;& \\
  3\alpha-3\gamma+1 \geq 0; & &3\beta-3\gamma+1 \geq 0;& \\
  3\alpha+\beta-6\gamma+2 \geq 0; & & \alpha+3\beta-6\gamma+2 \geq 0.&
  \end{align*}
  \item [Vertices: $(\alpha,\beta, \gamma)$]
    \begin{align*}
    \{0,1\}^2\times\{0\};& & (1,1,1);& &(0,0,1/3);&  \\
    (0,1,1/3);& & (1,0,1/3);& & (1/3,1,2/3);& \\
    (1,1/3,2/3).& & & & & \\
  \end{align*}
  \end{description}
  \vfill
  \end{minipage}
\end{figure}

    \begin{figure}[H]
  \centering
  \begin{minipage}[b]{0.4\textwidth}
      \begin{tikzpicture}[tdplot_main_coords, scale=2, >=stealth]   
    \draw[fill=gray,opacity=0.5] (0,0,0) -- (0,1,0) -- (0,1,1/4) -- (0,0,1/4) -- cycle;
    \draw[fill=gray,opacity=0.5] (0,0,0) -- (0,0,1/4) -- (1/3,0,1/2) -- (1,0,1/2) -- (1,0,0) -- cycle;
    \draw[fill=gray,opacity=0.5] (0,0,1/4) -- (0,1,1/4) -- (1/6,1,1/2) -- cycle;
    \draw[fill=gray,opacity=0.5] (0,0,1/4) -- (1/6,1,1/2) -- (1/2,1,3/4) -- (1/3,0,1/2) -- cycle;
    \draw[fill=gray,opacity=0.5] (1,1,1) -- (1/3,0,1/2) -- (1/2,1,3/4) -- cycle; 
     \draw[fill=gray,opacity=0.5] (1,1,1) -- (1/3,0,1/2) -- (1,0,1/2) -- cycle; 
    \draw[dashed] (1,0,0) -- (1,1,0);
    \draw[dashed] (0,1,0) -- (1,1,0);
    \draw[dashed] (1,1,1) -- (1,1,0);
        \draw[black] (0,0,0) node[below] {$(0,0,0)$};
    \draw[black] (1,1,1) node[above] {$(1,1,1)$};
    \draw[black] (1,0,0) node[left] {$(0,1,0)$};
    \draw[black] (0,1,0) node[right] {$(1,0,0)$};
  \end{tikzpicture} 
    \caption{The Constraint Polytope when $(d_1,d_2)=(2,4)$}
    \label{matlab_2_4}
  \end{minipage}
  \quad
  \begin{minipage}[b]{0.4\textwidth}
  {\bf Polytope in Parameter Space:}
  \begin{description}
  \item [Facets:]
  \begin{align*}
    \alpha, \beta,\gamma \geq 0;& &\alpha,\beta \leq 1;& \\
  \alpha-2\gamma+1 \geq 0; & &6\beta-4\gamma+1 \geq 0;& \\
  \alpha+6\beta-8\gamma+2 \geq 0; & & \alpha+3\beta-6\gamma+2 \geq 0.&
  \end{align*}
  \item [Vertices: $(\alpha,\beta, \gamma)$]
    \begin{align*}
    \{0,1\}^2\times\{0\};& & (1,1,1);& &(0,0,1/4);&  \\
    (0,1,1/2);& & (1,0,1/4);& & (1,1/2,3/4);& \\
    (1,1/6,1/2);& & (0,1/3,1/2). & & & \\
  \end{align*}
  \end{description}
  \vfill
  \end{minipage}
\end{figure}

    \begin{figure}[H]
  \centering
    \begin{minipage}[b]{0.4\textwidth}       \begin{tikzpicture}[tdplot_main_coords, scale=2, >=stealth]   
    \draw[fill=gray,opacity=0.5] (0,0,0) -- (1,0,0) -- (1,0,1/4) -- (0,0,1/4) -- cycle;
    \draw[fill=gray,opacity=0.5] (0,0,0) -- (0,0,1/4) -- (0,1,1/4) -- (0,1,0) -- cycle;
    \draw[fill=gray,opacity=0.5] (0,0,1/4) -- (1,0,1/4) -- (1,1/6,1/2) -- cycle;
    \draw[fill=gray,opacity=0.5] (0,0,1/4) -- (1,1/6,1/2) -- (1,1/2,3/4) -- cycle;
    \draw[fill=gray,opacity=0.5] (0,0,1/4) -- (1,1,1) -- (1,1/2,3/4) -- cycle;
    \draw[fill=gray,opacity=0.5] (0,0,1/4) -- (1,1,1) -- (1/2,1,3/4) -- cycle; 
    \draw[fill=gray,opacity=0.5] (0,0,1/4) -- (1/6,1,1/2) -- (1/2,1,3/4) -- cycle;
     \draw[fill=gray,opacity=0.5] (0,0,1/4) -- (1/6,1,1/2) -- (0,1,1/4) -- cycle;   
    \draw[dashed] (1,0,0) -- (1,1,0);
    \draw[dashed] (0,1,0) -- (1,1,0);
    \draw[dashed] (1,1,1) -- (1,1,0);
        \draw[black] (0,0,0) node[below] {$(0,0,0)$};
    \draw[black] (1,1,1) node[above] {$(1,1,1)$};
    \draw[black] (1,0,0) node[left] {$(0,1,0)$};
    \draw[black] (0,1,0) node[right] {$(1,0,0)$};
  \end{tikzpicture} 
  
    \caption{The Constraint Polytope when $(d_1,d_2)=(4,4)$}
    \label{matlab_4_4}
  \end{minipage}
\quad
  \begin{minipage}[b]{0.4\textwidth}
  {\bf Polytope in Parameter Space:}
  \begin{description}
  \item [Facets:]
  \begin{align*}
    \alpha, \beta,\gamma \geq 0;& &\alpha,\beta \leq 1;& \\
      6\alpha-4\gamma+1 \geq 0; & & 6\beta-4\gamma+1 \geq 0;& \\
  \alpha+2\beta-4\gamma+1 \geq 0; & &2\alpha+\beta-4\gamma+1 \geq 0;& \\
  \alpha+6\beta-8\gamma+2 \geq 0; & & 6\alpha+\beta-8\gamma+2 \geq 0.&
  \end{align*}
  \item [Vertices: $(\alpha,\beta, \gamma)$]
   \begin{align*}
     \{0,1\}^2\times\{0\};& & (1,1,1);&&(0,0,1/4);&  \\
   (0,1,1/4);&&(1,0,1/4);&& (1/2,1,3/4);& \\
   (1,1/2,3/4);&&(1/6,1,1/2);&&(1,1/6,1/2).&
  \end{align*}
  \end{description}
  \vfill
  \end{minipage}
\end{figure}

\begin{landscape}
\begin{figure}
\centering
\vspace{2.5 cm}
\begin{tikzpicture}[scale=1.1]
  \tikzset{venn circle/.style={draw,circle,minimum width=4cm,fill=#1,opacity=0.5}}
  \draw[very thick] (14,-3.5) -- (-3,-3.5);
  \draw[very thick] (-3,-3.5) -- (-3,3);
  \draw[very thick] (14,3) -- (-3,3);
  \draw[very thick] (14,3) -- (14,-3.5);
  \draw (-1,1) circle (1.5cm);
  \draw (1,1) circle (1.5cm);
  \draw (0,-.7) circle (1.5cm);
  \draw (2,-.7) circle (1.5cm);
  \draw (3,1) circle (1.5cm);
  \draw[dotted, line width=1pt]  (4,-.7) circle (1.5cm);
  \draw (12,-.7) circle (1.5cm);
  \draw (11,1) circle (1.5cm);
  \draw (10,-.7) circle (1.5cm);
  \draw[dotted, line width=1pt] (5,1) circle (1.5);
  \draw[dotted, line width=1pt] (9,1) circle (1.5);
  \draw[dotted, line width=1pt] (8,-.7) circle (1.5);
   \node at (-1.2,2.1) {\small $A_1$}; 
   \node at (1,2.1) {\small $A_3$}; 
   \node at (3,2.1) {\small $A_5$}; 
   \node at (-.2,-1.8) {\small $A_2$}; 
   \node at (2,-1.8) {\small $A_4$}; 
   \node at (12,-1.8) {\small $A_d$};
   \node at (10,-1.8) {\small $A_{d-2}$};
   \node at (11,2.1) {\small $A_{d-1}$};
   \node at (-1.2,1.2) {\small $\xi_1$};
    \node at (-.2,-1) {\small $\xi_2$};
    \node at (1,1.2) {\small $\xi_3$}; 
     \node at (2,-1) {\small $\xi_3$};
     \node at (11,1.2) {\small $\xi_2$};
     \node at (12,-1) {\small $\xi_1$};
     \node at (10,-1) {\small $\xi_3$};
     \node at (0,0.4) {\Large $\nu$}; 
     \node at (1,-.1) {\Large $\nu$}; 
     \node at (2,.4) {\Large $\nu$};
     \node at (11,-.1) {\Large $\nu$};
    \node at (-.8,0) {\scriptsize $\kappa-\nu$};
    \node at (11.8,.4) {\scriptsize $\kappa-\nu$};
    \node at (0,1.2) {\scriptsize $\mu-\nu$}; 
     \node at (2,1.2) {\scriptsize $\mu-\nu$}; 
     \node at (1,-.8) {\scriptsize $\mu-\nu$};
     \node at (11,-.8) {\scriptsize $\mu-\nu$};
      \node[rotate=60] at (.5,.2) {\tiny $\kappa-2\nu$}; 
     \node[rotate=-60] at (1.5,.2) {\tiny $\kappa-2\nu$}; 
     \node at (7,1.2) {\Large $\cdots$}; 
     \node at (6,-1) {\Large $\cdots$}; 
  \node[below] at (8,-2.5 ) {\Large $A_1^c \cap \cdots \cap  A_d^c$}; 
\end{tikzpicture}
\caption{Construction of Sets $A_i$, $i=1,\ldots, d$, for any $d$: 2-Dependent TDM of Proposition \ref{MA2}}
\label{2depsuffVenn}
\end{figure}
\end{landscape}

%
%

\end{document}